\documentclass[11pt, reqno]{amsart}
\usepackage{amssymb,latexsym,amsmath,amsfonts}
\usepackage[mathscr]{eucal}
\usepackage{xcolor}

\numberwithin{equation}{section}

\theoremstyle{definition}
\newtheorem{definition}{Definition}[section]
\newtheorem{example}[definition]{Example}

\newtheorem{remark}[definition]{Remark}

\theoremstyle{plain}
\newtheorem{theorem}[definition]{Theorem}
\newtheorem{lemma}[definition]{Lemma}
\newtheorem{proposition}[definition]{Proposition}
\newtheorem{corollary}[definition]{Corollary}

\newtheorem{result}[definition]{Result}

\newcommand{\scr}[1]{\mathscr{#1}}

\newcommand{\tends}{\rightarrow}

\newcommand{\Cpn}{\mathbb{C}^n}
\newcommand{\cplx}{\mathbb{C}}

\newcommand{\rea}{\mathbb{R}}

\newcommand{\natu}{\mathbb{N}}
\newcommand{\Ot}{\tilde{\Omega}}
\begin{document}

\title[Runge embeddings, approximations, and the Loewner PDE ]{Runge embeddings, approximation of biholomorphisms on Stein manifolds, and the Loewner PDE}
\author{Sushil Gorai and Gourab Paul}
\address{Department of Mathematics and Statistics, Indian Institute of Science Education and Research Kolkata,
Mohanpur, Nadia, West Bengal 741246, India} 
\email{sushil.gorai@iiserkol.ac.in}

\address{Department of Mathematics and Statistics, Indian Institute of Science Education and Research Kolkata,
Mohanpur, Nadia, West Bengal 741246, India}
\email{gp23rs051@iiserkol.ac.in}

\keywords{Runge embedding; Stein manifold with density property; approximation of biholomorphism; $(R,+)$-action; Loewner PDE}
\subjclass[2020]{Primary: 32E30, 32Q28, 32Q40, 32Q35, 32M17, 32M25, 32M05}

\begin{abstract}
We develop an extension-by-approximation principle for holomorphic Runge embeddings of increasing union of Stein manifolds into Stein manifolds with density property. The basic hypothesis is the existence, on each stage of the exhaustion, of a Runge isotopy which compresses the stage and whose terminal map extends holomorphically to the next stage. The resulting global embedding of the union may be chosen with Runge image, and every Runge embedding of a fixed stage can be approximated uniformly on compact subsets by the Runge embeddings of the union. 

We apply this principle to domains that are invariant underpositive time part of holomorphic $(R,+)$-actions, to Stein manifolds carrying a semicomplete holomorphic vector field with globally attracting fixed point. It also gives a Runge embedding of
$(\mathbb{C}^n\setminus \{z\in\cplx^n: f(z)=0\})\times \cplx$ in $\cplx^{n+1}$, which generalizes previous result of Runge embedding of $(\cplx^*)^n\times\cplx$ in $\cplx^{n+1}$. We also construct Stein globalization of an injective holomorphic semigroup action to holomorphic $(R,+)$-action. Finally, the abstract Loewner range of a Herglotz vector field is shown to admit a same-dimensional Runge embedding whenever the initial domain admits a Runge embedding into a Stein domain with density property; this yields a corresponding solution of the Loewner PDE with values in $\cplx^n$. We also give an example of non-Runge complete hyperbolic domain which admits $\mathbb{C}^n$-valued solution of the Loewner PDE.
\end{abstract}

\maketitle

\section{Introduction}\label{sec-intro}
It is well-known from the works of Remmert \cite{Remmert}, Narasimhan \cite{Narasimhan}, and Bishop \cite{Bishop} that any Stein manifold $X$ of dimension $n$ admits a proper holomorphic embedding into $\mathbb{C}^{2n+1}$. As an application of Gromov's Oka principle, Eliashberg and Gromov \cite{EliashbergGromov}, Sch{\"u}rmann \cite{Schurmann} proved that the minimal dimension to be $\lfloor\frac{3n}{2}\rfloor+1$ for $n\geq 2$. The question of embedding into $\cplx^n$ 
is deep and classical. In this paper we will consider embeddings a Stein manifold of dimension $n$ into $\cplx^n$ or a given Stein manifold with density of the same dimension.  More precisely,  
Our main question here is: {\em Under what conditions a Stein manifold of dimension $n$ has a Runge Stein embedding in a given Stein manifold $W$ of dimension $n$ with the density property?}. A characterization of existence of such embedding seems out of reach, even if $W=\cplx^n$. 

One of the motivation to study this question is the Andersen-Lempert theory, more specifically, the approximation of injective holomorphic maps by automorphisms.
In this paper, some sufficient conditions for existence of such Runge embedding are demonstrated and some applications, mainly in the context of Loewner theory, are provided. 
   Before stating the main result let us fix some notations. Let $\Omega\subset\tilde{\Omega}$ be two domains in a complex manifold $X$ of dimension $n$, and $W$ be a Stein manifold of the same dimension. Denote by $\mathcal{E}(\Omega,W)$ the space of all holomorphic embeddings $F:\Omega\rightarrow W$, endowed with the compact open topology.  Throughout the paper we use the term holomorphic embedding or biholomorphism into the codomain to mean an injective holomorphic map between two complex manifolds of same dimension. Recall that a pair $(\Omega,\tilde{\Omega})$ of domain in a complex manifold $X$ with $\Omega\subset\tilde{\Omega}$ is called a {\em Runge pair $($or $\Omega$ is Runge in $\tilde{\Omega})$} if any holomorphic function on $\Omega$ can be approximated by holomorphic functions of $\tilde{\Omega}$ locally uniformly on $\Omega$. We set
    \[ \mathcal{E}_{\mathcal{R}}(\Omega,W)=\{F\in \mathcal{E}(\Omega,W)|  \text{ $F(\Omega)$ is Runge in $W$} \},\]
      \[\mathcal{E}_{\mathcal{R}}(\tilde{\Omega},W)|_{\tilde{\Omega}}=\{F|_{\Omega}:F\in \mathcal{E}_{\mathcal{R}}(\tilde{\Omega},W) \}.\] 
      By $\overline{\mathcal{E}_{\mathcal{R}}(X,W)}|_{\Omega_s}$ we denote the closure of $\mathcal{E}_{\mathcal{R}}(X,W)|_{\Omega_s}$ in $\mathcal{E}(\Omega_s,W)$.
     A complex manifold $W$ is said to have the \textit{density property} if the Lie algebra generated by all $\mathbb{C}$-complete holomorphic vector fields on $W$ is dense in the Lie algebra of all holomorphic vector fields on $W$ with respect to the compact open topology. 
     This notion was introduced by Varolin in \cite{Varolin} and is now very active area of research (see \cite{ForstnericKutzschebauch} for a recent survey).
     Some examples of Stein manifolds that enjoys the density property are $\Cpn (n\geq 2)(\text{Anders{\' e}n-Lempert \cite{AndersenLempert}})$, $G\times \mathbb{C}$, where $G$ is a Stein Lie group (Varolin \cite{Varolin}). 

     Our first result is an extension-by-approximation principle for holomorphic Runge embeddings of increasing union of Stein manifolds into Stein manifolds with density property. More precisely:
   
    \begin{theorem}\label{Main theorem for embedding Stein manifolds}
    Let $X$ be a Stein manifold of dimension $n\geq 2$ and $\{\Omega_j\}_{j\in \mathbb{N}}$ be a family of non-empty Stein open connected subsets of $X$ such that 
    \[
    \Omega_1\subset \hdots \subset \Omega_j\subset\Omega_{j+1}\subset\hdots\subset \bigcup_{j\in \mathbb{N}}\Omega_j=X.
    \]
    Let $W$ be a Stein manifold of dimension $n$ with the density property with $\mathcal{E}_{\mathcal{R}}(\Omega_1,W)\neq \emptyset$.
    Assume that, for each $j\in \mathbb{N}$, there exists a continuous isotopy of injective holomorphic maps $H_j:[0,1]\times \Omega_j\rightarrow \Omega_j$ such that
    \begin{itemize}
        \item[(i)]   $H_j(0,z)=z,\ \forall z\in \Omega_j$,
        
       \item[(ii)] $(H_{j,t}(\Omega_j),\Omega_j)$ is a Runge pair for each $t\in [0,1]$, where $H_{j,t}=H_j(t,\cdot):\Omega_j\rightarrow\Omega_j$.
       
        \item[(iii)] $H_{j,1}:\Omega_j \rightarrow \Omega_j$ extends to an injective holomorphic map $\tilde{H}_{j,1}:\Omega_{j+1}\rightarrow\Omega_j$ such that $(\tilde{H}_{j,1}(\Omega_{j+1}),\Omega_j)$ is a Runge pair.
    \end{itemize}
     Then, there exists a holomorphic embedding $\Psi:X\rightarrow W$ such that $\Psi(X)$ is Runge in $W$.
    Moreover, for every $s\in \mathbb{N}$, \[\mathcal{E}_{\mathcal{R}}(\Omega_s,W)\subset \overline{\mathcal{E}_{\mathcal{R}}(X,W)}|_{\Omega_s}.  \]
\end{theorem}

 We note that the conclusion of Theorem~\ref{Main theorem for embedding Stein manifolds} has two parts: One concerns the existence of Runge and Stein embedding  and the other shows that, for each $j\in\natu$ any injective holomorphic map $F:\Omega_j\rightarrow W$ with $F(\Omega_j)$ Runge in $W,$ can be approximated by injective holomorphic maps from $X$ into $W$ having Runge image in $W$, locally uniformly in $X$. There are a series of application of Theorem~\ref{Main theorem for embedding Stein manifolds} in both the cases. We will now state them. Since the embedding and approximation properties are closely connected, we will not make any preference. The first few theorems that comes next are approximation theorems.
 
 For a complex manifold $M$ consider the set \[ Aut(M)=\{\Phi:M\rightarrow M| \ \Phi \ \text{is one-one, onto holomorphic map} \}.\] Anders{\'e}n and Lempert \cite{AndersenLempert} proved that, for a  star-shaped domain in $\cplx^n$, $\mathcal{E}_{\mathcal{R}}(\Omega,\Cpn)\subset\overline{\mathcal{E}_{\mathcal{R}}(\Cpn,\Cpn)}|_{\Omega}\subset\overline{Aut(\Cpn)}|_{\Omega}$ in $\mathcal{E}(\Omega,\Cpn)(n\geq 2)$. So, it is natural to ask : {\em On which domains $\Omega\subset \Cpn$ we have  $\mathcal{E}_{\mathcal{R}}(\Omega,\Cpn)\subset\overline{Aut(\Cpn)}|_{\Omega}$ in $\mathcal{E}(\Omega,\Cpn)$?} Works in this question is of current interest \cite{ChatterjeeGorai, Hamada, Forstneric2025}.
 The most general class of such domains so far are the domains that remains invariant under positive time flows of a holomorphic vector field of $\Cpn$ with a globally attracting fixed point on $\Cpn$ (see \cite{Forstneric2025}, \cite{ChatterjeeGorai}).  Here we show that even the attracting fixed point is not required. We state our result in terms of $(R,+)$-action.
\begin{theorem}\label{cor-Approximation on action invariant domains of Cn}
      Let $\theta:\mathbb{R}\times \Cpn\rightarrow\Cpn(n\geq 2)$ be a continuous action of $(\mathbb{R},+)$ on $\Cpn$ such that for each $t\in \mathbb{R}$, $\theta_t :\Cpn\rightarrow \Cpn $ is holomorphic. Let $\Omega$ be a Stein domain in $\Cpn$ such that $\Omega$ is invariant under $\theta_t \ (t\geq0)$ and each orbit $\{\theta_t(z):t\in \mathbb{R}\}\ (z\in \Cpn)$ intersects $\Omega$. Then $\mathcal{E}_{\mathcal{R}}(\Omega,\Cpn)\subset \overline{Aut(\Cpn)}|_{\Omega}$ in $\mathcal{E}(\Omega,\Cpn)$.
\end{theorem}

A generalization of Theorem~\ref{cor-Approximation on action invariant domains of Cn} to Stein manifolds is also possible; here it follows:
 
    \begin{theorem}\label{Embedding of Stein manifold using R-actions}
    Let $X$ be a Stein manifold of dimension $n\geq 2$ and $\theta:\mathbb{R}\times X\rightarrow X$ be a continuous action of $(\mathbb{R},+)$ such that $\theta_t:X\rightarrow X$ is holomorphic for each $t\in \mathbb{R}$. Let $\Omega(\neq \emptyset)$ be a Stein domain in $X$ such that $\theta_t(x)\in \Omega$ for each $x\in \Omega$ and $t\geq 0$. Let $W$ be a Stein manifold with the density property of dimension $n$. Then $\mathcal{E}_{\mathcal{R}}(\Omega,W)\subset\overline{\mathcal{E}_{\mathcal{R}}(X',W)}|_{\Omega}$ in $\mathcal{E}(\Omega, W)$, where  $X'=\bigcup_{t\geq 0}\theta_{-t}(\Omega)$.
\end{theorem}

We now discuss the applications of Theorem~\ref{Main theorem for embedding Stein manifolds} in the context of Runge embedding. It was an open question for sometime whether $\mathbb{C}^*\times \mathbb{C}$ admits a holomorphic Runge embedding in $\mathbb{C}^2$, which is equivalent to whether $\mathcal{E}_{\mathcal{R}}(\mathbb{C}^*\times \mathbb{C},\mathbb{C}^2)\neq \emptyset$. Origin of this question traces back to dynamics \cite{BedfordSmillie} by E. Bedford and J. Smillie. It was solved positively by Bracci, Raissy and Stens{\o}nes  in \cite{BracciRaissyStensones}. F. Forstneri\v c and E. F. Wold \cite{ForstnericWoldRungetubes} gave a different proof. Their approach yields a huge class of examples of non-trivial Runge domains in Stein manifolds with the density property.
For any Runge domain $\Omega$ in $\Cpn \times \mathbb{C}$ Serre \cite{Serre} proved that $H^q(\Omega,\mathbb{C})=0,\ q\geq n+1.$ It gives a cohomological obstruction on Runge embedding of $(\mathbb{C}^*)^{n+1}=(\Cpn\times \mathbb{C})\setminus\{(z_1,\dots,z_{n+1}):z_1z_2\dots z_{n+1}=0\}$ in $\mathbb{C}^{n+1}$ as $H^{n+1}((\mathbb{C}^*)^{n+1},\mathbb{C})\neq 0$. Therefore, 
 $(\mathbb{C}^n\times \mathbb{C})\setminus\{g=0\}$ does not admit a Runge embedding in $\mathbb{C}^{n+1}$ for an arbitrary $g\in \mathcal{O}(\mathbb{C}^{n+1})$. 
  Here we present: 
\begin{theorem}\label{cor-Runge embedding of Cn-A x C}
    Let $f\in \mathcal{O}(\Cpn)$ and $A=\{z \in \Cpn: f(z)=0\}$. Then $(\Cpn\setminus A)\times \mathbb{C}$ admits a holomorphic Stein and Runge embedding in 
    $\mathbb{C}^{n+1}$.
\end{theorem}
   
   \noindent Note that, for $n\geq2$, the normal bundle associated to a holomorphic embedding $\psi:\Cpn\setminus A \rightarrow \mathbb{C}^{n+1}$, which has $\mathcal{O}(\mathbb{C}^{n+1})$-convex image, is not, in general, trivial. Hence, the above theorem does not follow from \cite{ForstnericWoldRungetubes}. Although, if $n=1$, it directly follows from \cite[Corollary~1.2]{ForstnericWoldRungetubes}. 
   
   Next we find domains $\Omega\subset\Cpn$ for which $\mathcal{E}_{\mathcal{R}}(\Omega,\Cpn)\subset \overline{Aut(\Cpn)}|_{\Omega}$ in $\mathcal{E}(\Omega,\Cpn)$ can be shown by a different method. This allows us to get some other results on Runge embeddings and approximation of injective holomorphic maps. We first consider the Runge domains in $\Cpn$ with some intrinsic property on it to get our desired approximation results. 

\begin{theorem}\label{Starshaped}
    Let $\Omega$ be a domain in $\Cpn$, $n\geq 2$. Assume that there exists a continuous isotopy of biholomorphism $H: [0,1]\times \Omega\rightarrow \Omega$, such that $H(0,z)=z$ and $H$ satisfies that $H_1(\Omega)\subset D$, where $D$ is a star-shaped domain in $\Omega$. Then, $\mathcal{E}(\Omega,\Cpn)$ is path connected.\\
     Additionally, if we assume $H_t(\Omega)$ Runge in $\Cpn$ for each $t\in[0,1]$, then $\mathcal{E}_{\mathcal{R}}(\Omega,\Cpn)$ is path connected and $\mathcal{E}_{\mathcal{R}}(\Omega,\Cpn)\subset\overline{Aut(\Cpn)}|_{\Omega}$ in $\mathcal{E}(\Omega,\Cpn)$.
     \end{theorem}
     Note that, it differs from Theorem \ref{cor-Approximation on action invariant domains of Cn} as we do not need a global action on $\Cpn$ and invariant domains in this case.
    As a consequence of this result, we get several classes of examples of such domains $\Omega\subset\Cpn$ where approximation by automorphisms holds.
    
    Any  domain in $\mathbb{C}$ is simply connected if and only if it admits a $\mathbb{R}_+$-complete holomorphic vector field with a unique globally asymptotically stable equilibrium point if and only if it is Runge in $\mathbb{C}$, thanks to  Riemann mapping theorem and Runge approximation theorem.
    But for $n\geq 2$, a domain in $\Cpn$ that admits a $\mathbb{R}_+$-complete holomorphic vector field with a unique globally asymptotically stable equilibrium point,
      need not  be Runge in $\Cpn$. 
      It is also not known whether they are biholomorphic to  Runge domains in $\Cpn$. We get an answer from the following theorem, more generally for Stein manifolds.
    \begin{theorem}\label{Runge embedding of Stein manifolds which admits a R+complete holomorphic vector field $V$ which has a unique equilibrium point that is globally asymptotically stable}
    Let $X$ be a Stein manifold of dimension $n\geq 2$. $X$ admits a $\mathbb{R}_+$-complete holomorphic vector field $V$ which has a unique equilibrium point that is globally asymptotically stable. Then $X$ admits a holomorphic Stein and Runge embedding in $\Cpn$. 
\end{theorem}

The following theorem gives us another Runge embedding theorem for an increasing union of Stein manifolds into $\Cpn$. It generalizes  \cite[Theorem~3.4]{ArosioBracciWold}, \cite[Theorem~4.5]{Hamada}, \cite[Theorem~5.1,V5]{ChatterjeeGorai} substantially.
\begin{theorem}\label{Runge embedding of increasing sequence of Stein domains}
    Let $X$ be a complex manifold of dimension $n\geq 2$. Let $\{\Omega_j\}_{j\in \mathbb{N}}$ be a family of nonempty Stein domains of $X$ such that 
    \[
    \Omega_1\subset \hdots \subset \Omega_j\subset \hdots \subset\cup_{j\in \mathbb{N}}\Omega_j=X.
    \]
    Suppose that 
    \begin{itemize}
  \item [(i)]  each $(\Omega_j,\Omega_{j+1})$ is a Runge pair $(j\in \mathbb{N})$.
  
    \item [(ii)] each $\Omega_j$ admits a $\mathbb{R}_+$-complete holomorphic vector field with a unique equilibrium point which is globally asymptotically stable on $\Omega_j$ $(j\in \mathbb{N})$.
    \end{itemize}
    Then, $X$ admits a Stein and Runge embedding into $\Cpn$; and furthermore, $\mathcal{E}_{\mathcal{R}}(\Omega_j,\Cpn)\subset\overline{\mathcal{E}_{\mathcal{R}}(X,\Cpn)}|_{\Omega_j}$ in $\mathcal{E}(\Omega_j,\Cpn)$ for any $j\in \mathbb{N}$.
    \end{theorem}
    \begin{remark}
    Here we do not need to consider isotopies as in Theorem \ref{Main theorem for embedding Stein manifolds} rather we need an intrinsic property of each Stein manifold in terms of vector fields given in assumption $(ii)$, and relative Runge property given in assumption $(i)$.
    \end{remark}
Next we present one more theorem of similar flavour. Here the method of proof of above theorem does not work.
      \begin{theorem}\label{thm-semigroupAction}
         Let $X$ be a Stein manifold of dimension $n\geq 2$ and $\theta:\mathbb{R}_+\times X\rightarrow X$ be a continuous semigroup action of $(\mathbb{R}_+,+)$ such that for each $t\geq 0$, $\theta_t=\theta(t,\cdot):X\rightarrow X  $ is holomorphic. Let $\Omega\subset X$ be a Stein domain that is invariant under $\theta_t \ (\forall t\geq0)$ and each orbit $\{\theta_t(z):t\in \mathbb{R}_+\}(z\in X)$ intersects $\Omega$. Let $W$ be a Stein manifold of dimension $n$ with the density property. Then  $\mathcal{E}_{\mathcal{R}}(\Omega,W)\subset \overline{\mathcal{E}_{\mathcal{R}}(X,W)}|_{\Omega}$ in $\mathcal{E}(\Omega,W)$. 
        
     \end{theorem}

    In 1923, C. Loewner \cite{Loewner} introduced the classical Loewner theory for the unit disc in $\mathbb{C}$. Later it was developed with contributions of P. P. Kufarev \cite{Kufarev} in 1943 and C. Pommerenke \cite{Pommerenke} in 1965. In higher dimensions, J. A. Pfaltzgraff extended the basic theory to $\mathbb{C}^n$ in \cite{Pfaltzgraff1}, \cite{Pfaltzgraff2}. In \cite{BracciContrerasMadrigal2}, F. Bracci, M. D. Conteras and S. D{\'i}az-Madrigal proved that on a complete hyperbolic domain $D\subset\Cpn$ given a Herglotz vector field of order $d\geq 1$ (Definition \ref{Def-Herglotz vector field}) $G:D\times \mathbb{R}_+\rightarrow \Cpn$ there always exists a solution $(\phi_{s,t})_{0\leq s\leq t}$ to the Loewner ODE (see also \cite{ArosioBracci}),
   \begin{equation*}
       \frac{\partial\phi_{s,t}}{\partial t}(z)=G(\phi_{s,t}(z),t) \ \ a.e. \ t\in [s,+\infty).
   \end{equation*} 
   
   An open question in this topic is: {\em On a complete hyperbolic domains $D\subset \Cpn(n\geq 2)$ given a Herglotz vector field $G(z,t),$ does the Loewner PDE
   \begin{equation*}
     \frac{\partial f_{t}}{\partial t}(z)=-df_t(z)G(z,t),   \ a.e. \  t\geq 0, \ \forall z\in D,   
   \end{equation*}
   admits a $\Cpn$-valued solution?}
   Partial answers are given in \cite{GrahamHamadaKohr},\cite{DurenGrahamHamadaKohr},\cite{GrahamHamadaKohrKohr},\cite{Voda},\cite{Arosio3},\cite{Arosio2},\cite{Arosio1}. 
L. Arosio, F. Bracci, H. Hamada and G. Kohr in \cite{ArosioBracciHamadaKohr} showed that the Loewner PDE on complete hyperbolic manifolds always admits an (essentially unique) abstract univalent solution with values in a complex manifold, say $R$, of the same dimension. It made the problem a holomorphic embedding problem of $R$ into $\mathbb{C}^{dimR}$. 
Arosio, Bracci and Wold \cite{ArosioBracciWold} first used \textit{Anders{\'e}n-Lempert} theory to solve the Loewner PDE on star-shaped domains. For other general results using \textit{Anders{\'e}n-Lempert} theory one can see \cite{Hamada}, \cite{ChatterjeeGorai}, \cite{Forstneric2025}. The most general class of examples of such complete hyperbolic domains so far were domains invariant under positive time flows of a holomorphic vector field of $\Cpn$ which has a globally attracting fixed point. This follows from \cite{Forstneric2025}. 
Here we present: 
\begin{theorem}\label{Solution of Loewner on holomorphically convex domains in density property}
Let    $D\subset\Cpn$, $n\geq 2$, be a complete hyperbolic domain. Let $G:D\times \mathbb{R}_+\rightarrow \Cpn$ be a Herglotz vector field of order $d\in[1,+\infty]$. Assume that $D$ is biholomorphic to a Runge domain in $W$, where $W\subseteq \Cpn$ is a Stein domain with the density property.\\
    Then there exists a family of univalent mappings $f_t:D\rightarrow \Cpn \ (t\geq 0)$, which solves the Loewner PDE
\begin{equation}\label{Loewner PDE1}
      \frac{\partial f_{t}}{\partial t}(z)=-df_t(z)G(z,t),   \ a.e. \  t\geq 0, \ \forall z\in D,  
     \end{equation}
     Moreover, $R=\bigcup_{t\geq 0}f_t(D)$ is a Runge and Stein domain in $W$ and any other solution to (\ref{Loewner PDE1}) with values in $\Cpn$ is of the form $(\Phi\circ f_t)$ for a suitable holomorphic mapping $\Phi:R\rightarrow \Cpn$.
\end{theorem}
Note that, if $D$ is a complete hyperbolic Runge domain in $\Cpn$ then it satisfies the assumption of the above theorem for $W=\Cpn$ for $n\geq 2$, and the Loewner PDE always admits an $\Cpn$-valued solution on $D$.  Now the following corollary is an immediate consequence of Theorem \ref{Solution of Loewner on holomorphically convex domains in density property} and Theorem \ref{Runge embedding of Stein manifolds which admits a R+complete holomorphic vector field $V$ which has a unique equilibrium point that is globally asymptotically stable}.

\begin{corollary}\label{cor-Solution of Loewner on complete hyperbolic domains that admits special vector fields}
Let $D$ be a complete hyperbolic domain in $\Cpn$. Let $G:D\times \mathbb{R}_+\rightarrow \Cpn$ be a Herglotz vector field of order $d\in[1,+\infty]$. Assume that there exists $t_0\geq 0$ such that $G(\cdot,t_0)$ has a unique globally asymtotically stable equilibrium point.
    Then there exists a family of univalent mappings $f_t:D\rightarrow \Cpn \ (t\geq 0)$, which solves the Loewner PDE
\begin{equation*}
      \frac{\partial f_{t}}{\partial t}(z)=-df_t(z)G(z,t),   \ a.e. \  t\geq 0, \ \forall z\in D,  
     \end{equation*}
\end{corollary}
 For $n\geq 2$ we give the first example of a non-Runge domains in $\Cpn$ where the Loewner PDE admits a $\Cpn$-valued solution (see Example \ref{Example-Solution of Loewner on non-Runge domains} for details).
\begin{remark}
    For $n=1$ any complete hyperbolic Runge domain is simply connected, hence biholomorphic to the unit disc in $\mathbb{C}$. It is known \cite{ContrerasMadrigalGumenyuk-disc} that the Loewner PDE on the unit disc admits a $\mathbb{C}$-valued solution.
\end{remark}

Some ideas about the proofs: our proof of Theorem~\ref{Main theorem for embedding Stein manifolds} mainly depends on {\em  Anders{\'e}n-Lempert theory} and some union embedding technique for an increasing sequence of complex manifolds. We can break the idea of the proof into the following steps,
\begin{itemize}
    \item[(a)] In the first step we shall use Anders{\'e}n-Lempert theory to show $\mathcal{E}_{\mathcal{R}}(\Omega_j,W)\subset \overline{\mathcal{E}_{\mathcal{R}}(\Omega_{j+1},W)}|_{\Omega_j}$ in $\mathcal{E}(\Omega,W)$, for each $j\in \mathbb{N}$.  
    \item[(b)] Then we shall find an injective holomorphic map $\Psi:X\rightarrow W$, as a limit of a sequence $\{\Psi_j\}_{j\in \mathbb{N}}$ , where $\Psi_j\in \mathcal{E}_{\mathcal{R}}(\Omega_j,W).$
    \item[(c)] Next, we shall construct a family $\{L_j\}_{j\in \mathbb{N}}$ of compact subsets of $\Phi(X)$ such that for any compact $L\subset \Phi(X)$ there exists a $j_0$ such that $L\subset L_{j_0}$. Also, for any $L_j$, and any holomorphic function in its neighborhood can be approximated by holomorphic functions of $W$ uniformly on $L_j$. 
    \item[(d)] At the end, for any $F\in \mathcal{E}_\mathcal{R}(\Omega_s,W)$, we shall find a sequence $\{F_j\}_{j\in \mathbb{N}}$ in $\mathcal{E}_\mathcal{R}(X,W)$ that converges to $F$ locally uniformly on $\Omega_s$.
\end{itemize}
We shall use Proposition \ref{Continuous isotopy and approximation of Y valued biholomorphisms} to prove step $(a)$, and Theorem \ref{Union embedding in domain with Density} to complete the proof. 

\noindent Our proof of Theorem~\ref{cor-Approximation on action invariant domains of Cn} and Theorem~\ref{Embedding of Stein manifold using R-actions} uses Theorem~\ref{Main theorem for embedding Stein manifolds}. The $(R,+)$-action enables us to use Docquier-Grauert theorem to bypass Runge pair conditions and also helps us to construct the required homotopy. Theorem~\ref{cor-Runge embedding of Cn-A x C} needs some explicit construction along with Theorem~\ref{Main theorem for embedding Stein manifolds}. The proofs of Theorems~\ref{Main theorem for embedding Stein manifolds}, \ref{cor-Approximation on action invariant domains of Cn}, and \ref{Embedding of Stein manifold using R-actions}, \ref{cor-Runge embedding of Cn-A x C} are presented in Section~\ref{sec-main}.
\smallskip

\noindent We explicitly construct a Forstneri\v{c}-Rosay homotopy for approximation of biholomorphism in our proof of Theorem~\ref{Starshaped}. Along with that we prove Theorem~\ref{Runge embedding of Stein manifolds which admits a R+complete holomorphic vector field $V$ which has a unique equilibrium point that is globally asymptotically stable} and 
Theorem~\ref{Runge embedding of increasing sequence of Stein domains} in Section~\ref{sec-approx}. Proofs of these theorem uses $R_+$-action crucially; a technical lemma (Lemma~\ref{properties of invariant open exhaustion}) is presented which collects the properties of the domains present in the above theorems.
\smallskip

\noindent Proof of Theorem~\ref{thm-semigroupAction} needed to transfer the semigroup action into a group action. This is done in Theorem~\ref{Embedding of Steins having R_+ action into Steins having R actions} in Section~\ref{sec-transferSemigroup}.
\smallskip

\noindent In our proof of Theorem~\ref{Solution of Loewner on holomorphically convex domains in density property} our main theorem Theorem~\ref{Main theorem for embedding Stein manifolds} is used crucially. Final section is devoted to this.


\section{Preliminaries}\label{sec-prelims}
Let $M$ be a complex manifold and $K\subset M$ be a compact subset of $M$. Let $\mathcal{O}(M)=\{f:M\rightarrow \mathbb{C}: f\text{ is holomorphic on $M$}\}$, and
\[\mathcal{O}(K)=\{f:K\rightarrow\mathbb{C}: f\text{ is holomorphic in a neighborhood of $K$ in $M$}\},\]\[\mathcal{O}(M)|_{K}=\{f|_K:f\in \mathcal{O}(M)\}.\]

 Although the following lemma is a well-known result we write a proof as we have often used it throughout the paper. 
\begin{lemma}\label{Rungepair biholomorphic invariant}
   Let $\tilde{\Omega}, M$ be two complex manifold of the same dimension. Let $\Omega $ be an open set in $\tilde{\Omega}$ such that $(\Omega,\tilde{\Omega})$ is a Runge pair. Let $\Phi:\tilde{\Omega}\rightarrow M$ be any injective holomorphic map. Then $(\Phi(\Omega),\Phi(\tilde{\Omega}))$ is a Runge pair if and only if $(\Omega,\tilde{\Omega})$ is a Runge pair. Furthermore, If $K$ is a compact set in $\tilde{\Omega}$ such that  $\mathcal{O}(\tilde{\Omega})|_K$ is dense in $\mathcal{O}(K)$ then $\mathcal{O}(\Phi(\tilde{\Omega}))|_{\Phi(K)}$ is dense in $\mathcal{O}(\Phi(K))$.
\end{lemma}
\begin{proof}
    Take $f\in \mathcal{O}(\Phi(\Omega))$. Then $f\circ \Phi\in \mathcal{O}(\Omega)$.
    Since $(\Omega,\tilde{\Omega})$ is a Runge pair, there exists a sequence $\{f_k\}_{k\in \mathbb{N}}$ in $\mathcal{O}(\tilde{\Omega})$ such that $f_k$ converges to $f\circ \Phi$ locally uniformly on $\Omega$.
    This shows $f_k\circ \Phi^{-1}\in \mathcal{O}(\Phi(\Omega))$ and $f_k\circ  \Phi^{-1}$ converges to $f$ locally uniformly on $\Phi(\Omega)$.\\
    Assume that $(\Phi(\Omega),\Phi(\tilde{\Omega}))$ is a Runge pair. We know that $\Phi^{-1}:\Phi(\tilde{\Omega})\rightarrow \tilde{\Omega}$ is an injective holomorphic map. Hence from the first part $(\Phi^{-1}(\Phi(\Omega)),\Phi^{-1}(\Phi(\tilde{\Omega})))=(\Omega,\tilde{\Omega})$ is a Runge pair.\\
    The same proof works for the second statement.
\end{proof}



The following result is the central result of Anders{\'e}n-Lempert theory (discovered in \cite{AndersenLempert}). This $\mathcal{C}^p$-version $(p\geq 0)$ is proved by Forstneri\v c in \cite{Forstnericcont} (see also \cite{ForstnericRosay}).

\begin{result}\cite[Theorem~1.1]{Forstnericcont}\label{Forstneric Rosay continuous isotopy}
  Let $\Omega \subset \Cpn (n\geq 2)$ be a Runge domain, and let $B$ be the closed unit ball in $\mathbb{R}^k$. Assume that $F: B\times  \Omega\rightarrow \Cpn$ is a mapping of class $\scr{C}^p(0\leq p<\infty)$ such that for each $x\in B$, $F_x=F(x,\cdot):\Omega\rightarrow \Cpn$ is a biholomorphic mapping onto a Runge domain $\Omega_x\subset \Cpn$, and the map $F_0$ can be approximated by automorphisms of $\Cpn$, uniformly on compacts of $\Omega$. Then for each compact set $K\subset\Omega$ and each $\epsilon>0$ there exists a smooth map $\Phi: B\times \Cpn\rightarrow\Cpn$ such that $\Phi_x=\Phi(x,\cdot)$ is a holomorphic automorphism of $\Cpn$ for every $x\in B$ and $||F-\Phi||_{\scr{C}^p(B\times K)}<\epsilon $. 
\end{result}


 The following version on manifolds with the density property, has been stated by Varolin \cite{Varolin}. We take the following continuous version from \cite{ForstnericKutzschebauch}. 
\begin{result}\cite[Theorem~1.3]{ForstnericKutzschebauch}\label{Continuous isotopy in Manifold with the density}
    Let $X$ be a Stein manifold with the density property. Let $\Phi_t:\Omega\rightarrow \Omega_t=\Phi_t(\Omega)\subset X$ is a continuous isotopy of biholomorphic maps between Stein and Runge domains in $X$ with $\Phi_0=Id_{\Omega}$, then $\Phi_1$ can be approximated uniformly on compacts in $\Omega$ by holomorphic automorphisms of $X$.
\end{result}
The following lemma is an easy corollary to Result \ref{Continuous isotopy in Manifold with the density}. We will use it in the proof of Proposition \ref{Continuous isotopy and approximation of Y valued biholomorphisms}.
\begin{lemma}\label{Continuous isotopy in Manifold with the density-modified}
    Let $X$ be a Stein manifold with the density property and  $H_t :\Omega \rightarrow H_t(\Omega)\subset  X\ (t\in [0,1])$ be a continuous isotopy of biholomorphism into $X$. Also assume that for each $t\in[0,1]$, $H_t(\Omega)$ is Stein and Runge domain in $X$ with $H_0=Id_{\Omega}$. Then $H_1^{-1}$ can be approximated by holomorphic automorphisms of $X$ locally uniformly on $H_1(\Omega)$.
\end{lemma}

\begin{proof}
     Take the isotopy $G:[0,1]\times H_1(\Omega) \rightarrow X$ given by $G(t,w)=H_{1-t}\circ H_1^{-1}(w)$.
     Clearly, $G$ is a continuous and $G_0(w)=w, \ \forall w\in H_1(\Omega)$. For each $t\in [0,1]$, $G_t(H_1(\Omega))=H_{1-t}(\Omega)$, which is Stein and Runge in $X$. Now we can use Result \ref{Continuous isotopy in Manifold with the density} to conclude that  $G_1=H_1^{-1}$ can be approximated by holomorphic automorphisms of $X$ locally uniformly on $H_1(\Omega)$.
\end{proof}

    The following approximation theorem is also due to Anders{\'e}n and Lempert \cite{AndersenLempert}.
    \begin{result}\cite[Theorem~2.1]{AndersenLempert}\label{Andersen-Lempert Star theorem}
        Let $D\subset \Cpn$ be a star-shaped domain in $\Cpn$. Then any biholomorphic map $F:D\rightarrow\Cpn$ such that $F(D)$ is Runge in $\Cpn$ can be approximated by automorphisms of $\Cpn$ locally uniformly on $D$.
    \end{result}

Let $Hol(X,X)$ denote the space of all holomorphic self maps on a complex manifold $X$ endowed with the compact open topology. The following result is from \cite{Abate}. The semigroup is taken to be continuous by definition in that paper.
\begin{result}\cite[Proposition~1]{Abate}\label{injectivity of continuous semigroup action}
    Let $\Phi: \mathbb{R}_+\rightarrow Hol(X,X)$ be a one-parameter semigroup on a complex manifold $X$. Then $\Phi_t$ is injective for all $t\geq 0$.
\end{result}
Now, we mention a lemma from \cite{ArosioBracciHamadaKohr}.
 \begin{result}\cite[Corollary~3.2]{ArosioBracciHamadaKohr}\label{eventual containment of compacts}
     Let $U$ be an open set in $\Cpn$ and $\{F_k\}_{k\in \mathbb{N}}$ $(F_k:U\rightarrow \Cpn)$ be a sequence of injective holomorphic maps. Assume that $F_k$ converges to another injective holomorphic map $F:U\rightarrow\Cpn$. Then for any compact set $L\subset F(U)$ there exists an integer $k(L)$ such that $L\subset F_k(U)$ for all $k\geq k(L)$.
 \end{result}
  Next we state and prove Result \ref{eventual containment of compacts} for complex manifolds. This will be used in our proofs of Theorem \ref{Union embedding in domain with Density}, Lemma \ref{Hurwitz-non-degenerate limit of injective}, and Theorem \ref{Checking Runge pairs for R-actions}.
  \begin{lemma}\label{eventual containment of compacts for manifolds}
     Let $M,N$ be complex manifolds of the same dimension. Let $\{F_j\}_{j\in \mathbb{N}}$ $(F_j:M\rightarrow N)$ be a sequence of injective holomorphic maps. Assume that $F_j$ converges to another injective holomorphic map $F:M\rightarrow N$. Then for any compact set $L\subset F(M)$ there exists an integer $j(L)$ such that $L\subset F_j(M)$ for all $j\geq j(L)$.
 \end{lemma}
 \begin{proof}
     Let $b\in L$. We choose a chart $(V,\psi)$ of $N$ such that $b\in V,$ and another chart $(U,\phi)$ of $M$ such that $U\Subset M $ and \[b\in F(\bar{U})\subset V.\]
By assumption,  $F_j$ converges to $F$ uniformly on $\bar{U}$. So that there exists a $j_0\in \mathbb{N}$ such that $F_j(U)\subset V, \ \forall j\geq j_0$.
     Here $\psi\circ F_j\circ \phi^{-1}:\phi(U)\rightarrow \psi(V)$ $(j\geq j_0)$
is a sequence of injective holomorphic maps converging to  $\psi\circ F\circ \phi^{-1}:\phi(U)\rightarrow \psi(V)$. Let us choose a neighborhood $V_b\Subset F(U)$ of $b$ such that $\psi(\bar{V}_b)\subset \psi(V)$. Using Result \ref{eventual containment of compacts}, there exists a $k(b)\in \mathbb{N}$ such that $\psi(\bar{V}_b)\subset(\psi\circ F_j\circ \phi^{-1})(\phi(U))\Rightarrow \bar{V}_b\subset F_j(U) , \ \forall j\geq k(b)$.  Here, $\{V_b\}_{b\in L}$ is an open cover for $L$. 
Since $L$ is compact, there exists a finite set $\{b_1,\dots,b_m\}\subset L$ such that $L\subset \cup^m_{j=1}V_{b_j}$, and \[L\subset F_j(M), \ \forall j\geq max\{k(b_j):1\leq j\leq m\}.\]
This proves our lemma.
\end{proof}
Let $M,N$ be two complex manifolds of the same dimension. Let $F:M\rightarrow N$ be a holomorphic map and $p\in M$. We say that $F$ is {\em non-degenerate at $p$} if the differential $DF_p:T_pM\rightarrow T_{F(p)}N$ is an isomorphism.
 The following lemma  follows from Hurwitz's lemma. Since we did not find an explicit reference for it, we provide a short proof. We will use it in Theorem \ref{Union embedding in domain with Density}.
\begin{lemma}[Hurwitz]\label{Hurwitz-non-degenerate limit of injective}
    Let $M,N$ be two complex manifolds of the same dimension and $M$ is connected. Let $\{F_j\}_{j\in \mathbb{N}}$ $(F_j:M\rightarrow N)$ be a sequence of injective holomorphic maps converging locally uniformly to $F:M\rightarrow N$, and $F$ is non-degenerate at $p\in M$. Then $F$ is an injective holomorphic map on $M$.
\end{lemma}
\begin{proof}
We first show $F$ is non-degenerate at every point of $M$. Let us denote \[Q=\{z\in M: F \text{ is non-degenerate at z}\}.\] It is well known that $Q$ is open. Since each $F_j$ is injective, from Hurwitz's lemma it follows  that $Q$ is also closed. Since $M$ is connected and $Q\neq \emptyset$, we get that $M=Q$.

  On the contrary, let us assume that $F$ is not injective. Then there exists $a,b\in M$ such that $a\neq b$, and $F(a)=F(b)=w.$ Since $F$ is non-degenerate at $a$ and $b$, we can choose neighborhoods $U_a$ and $U_b$ of $a$ and $b$ respectively such that $\bar{U}_a\cap\bar{U}_b=\emptyset$ and $F|_{U_a}$, $F|_{U_b}$ are injective. Take a neighborhood $V_w$ of $w$ such that $w\in \bar{V}_w\subset F(U_a)\cap F(U_b)$. Now $\bar{V}_w\subset F(U_a)$ and $F_j|_{U_a}$ converges to $F|_{U_a}$ uniformly on $U_a$. Since $F|_{U_a}$ is injective, by Lemma \ref{eventual containment of compacts for manifolds}, we get that there exists a $N(a)\in \mathbb{N}$ such that $V_w\subset F_j(U_a),\ \forall j\geq N(a).$ Similarly, there exists a $N(b)\in \mathbb{N}$ such that $V_w\subset F_j(U_b), \ \forall j\geq N(b)$. Therefore we obtain that \[ V_w\subset F_N(U_a)\cap F_N(U_b), \ \text{ where } N=max\{N(a),N(b)\}\]
  Since $U_a\cap U_b=\emptyset$, it contradicts the fact that $F_N$ is injective on $M$. Hence $F:M\rightarrow N$ must be injective.
\end{proof}
Brickman introduced the definition of a $\Phi$-like domain in \cite{Brickman}, as a generalization of starlike and spirallike domains in $\cplx$. Gurganus defined it for higher dimensions in \cite{Gurganus}. In addition, one can find a weaker definition in \cite[Definition 4.22]{HamadaIancuKohr}, and we need this to state the following two results.
\begin{definition}\label{Definition of phi like domains}
    Let $\Omega\subset \Cpn$ be a domain. If $0 \in \Omega$ and there exists a holomorphic mapping $\Phi$ from $\Omega$ into $\Cpn$ such that $\Phi(0)=0$, and, for every $z\in \Omega$, the initial value problem 
    \begin{equation}
        \frac{\partial w}{\partial t}(z,t)=-\Phi(w(z,t)), \ \ t\geq 0, \ \ w(z,0)=z,
    \end{equation}
    has a solution $w(z,\cdot)$ on $[0,\infty)$ such that $w(z,t) \in \Omega$, $t \geq 0$, and $w(z,t) \rightarrow 0$, as $t\rightarrow \infty$, then we say that $\Omega$ is {\em $\Phi$-like domain}.
\end{definition}
We mention the following two results from \cite{HamadaIancuKohr}. If $D$ is invariant under $w_t$, for each $t\geq 0$, then $D$ is $\Phi|_D$-like. Using this fact we add the equivalent part in Result \ref{Rungepairphilikedomains}. 
\begin{result}\cite[Theorem~1.1]{HamadaIancuKohr}\label{Rungepairphilikedomains}
    Let $E\subset \Cpn$ be a $\Phi$-like domain.  Let $w$ be the semigroup on $E$ generated by $\Phi$. If $D\subseteq E $ is $\Phi|_{D}$-like then $(D,E)$ is a Runge pair. It is equivalent to say if $D$ is invariant under $w_t$, for each $t\geq 0$, then $(D,E)$ is a Runge pair. 
\end{result}
\begin{result}\cite[Corollary~4.12]{HamadaIancuKohr}\label{Rungepairinsemigroups}
  Let $E\subset \Cpn$ be a $\Phi$-like domain. Let $w$ be the semigroup on $E$ generated by $\Phi$. Then $(w_t(E),E)$ is a Runge pair, for all $t \geq 0$.  
\end{result}
As an analogue to $\Phi$-like domains on manifolds we define the following (see also \cite{CG-manifold}).
\begin{definition}\label{Definition- Manifolds that admits a vector field globally asymptotically stable}

Let $M$ be a complex manifold. We say that {\em $M$ admits a $\mathbb{R}_+$-complete holomorphic vector field $V$ with a unique equilibrium point $p$ which is globally asymptotically stable on $M$} if \begin{enumerate}
    \item  the initial value problem $\frac{\partial \theta}{\partial t}(t,z)= V(\theta(t,z)), \ t\geq 0,\ \theta(0,z)=z\in M$, has a solution for all $t\in[0,\infty)$ and for all $z\in M$.
    \item $\underset{t\rightarrow \infty}{lim}\theta(t,z)=p$ for each $z\in M$.
    \item For any neighborhood $V_p$ of $p$ in $M$, there exists a neighborhood $U_p$ of $p$ such that $\theta(t,z)\in V_p, \ \forall t\geq 0, \ \forall z\in U_p$. 
\end{enumerate} 
 \end{definition}
 The map $\theta:\mathbb{R}_+\times M\rightarrow M$ in condition $(1)$ is called {\em flow map } of the vector field $V$. It is well-known that it satisfies $\theta(t+s,z)=\theta(t,\theta(s,z)), \forall t,s\geq 0, \ z\in M .$


The following definition is taken from \cite{DocquierGrauert}. It was introduced by Docquier and Grauert. We shall use it in the proof of Theorem \ref{Checking Runge pairs for R-actions}.
\begin{definition}\cite[Definition~20]{DocquierGrauert}\label{definition-semicontinuously holomorphically extendable}
    
    Let $M$ be a non-empty open subset of a complex manifold $\tilde{M}$. Then we say that $M$ is \textit{semicontinuously holomorphically extendable} to $\tilde{M}$ by means of a family $(M_t)_{0\leq t\leq 1}$ of non empty open subsets of $\tilde{M}$ if the following hold:
    \begin{enumerate}
        \item $M_t$ is a Stein manifold for all $t$ in a dense subset of $[0,1]$,
        \item $M_0=M$ and $\bigcup_{0\leq t\leq 1}M_t=\tilde{M}$,
        \item $M_s\subset M_t$ for all $0\leq s< t\leq 1$,
        \item $\bigcup_{0\leq t<t_0}M_t$ is a union of connected components of $M_{t_0}$ for $0<t_0\leq 1$,
        \item $M_{t_0}$ is a union of connected components of interior part of $\bigcap_{t_0<t\leq 1}M_t$ for $0\leq t_0<1$.
    \end{enumerate}
\end{definition}
We take the following result from \cite{DocquierGrauert}. This is due to Docquier and Grauert.
\begin{result}\cite[Satz 17-Satz 19]{DocquierGrauert}\label{Runge pair by semicontinuous extension}
    Let $M$ is a non-empty Stein open subset of a Stein manifold $\tilde{M}$. If $M$ is semicontinuously holomorphically extendable to $\tilde{M}$, then $(M,\tilde{M})$ is a Runge pair. 
\end{result}
The following result is well-known. This was shown by K. Stein in \cite{Stein} (see also \cite{DocquierGrauert}).
\begin{result}\cite[Satz~1.2,1.3]{Stein}\label{Stein property of increasing union}
    Let $M$ be a complex manifold. Assume that $M=\bigcup_{j\in \mathbb{N}}M_j$, each $M_j$ is a Stein domain in $M$, and for each $j\in \mathbb{N}$ we have  $M_j\subset M_{j+1}$, and $(M_j,M_{j+1})$ is a Runge pair . Then $M$ is a Stein manifold. Also $(M_j,M)$ is a Runge pair, for each $j\in \mathbb{N}$. 
\end{result}


 We now mention some results from literature to show the existence of $\Cpn$-valued solutions of the Loewner PDE on complete hyperbolic domains biholomorphic to Runge domains. Our survey here is mainly from  \cite{ArosioBracciWold}.

Let $D\subset \Cpn$ be a domain. A holomorphic vector field $H$ on $D$ is called \textit{semicomplete$ ($or $\mathbb{R}_+$-complete)} if the Cauchy problem \begin{equation}
    \frac{d}{ds}w(s)=H(w(s)), \ w(0)=z_0
\end{equation}
has a solution $w$ defined for all $s\in [0,+\infty)$ and for all $z_0\in D$.

 \begin{definition}\label{Def-Herglotz vector field}
    Let $D\subset \Cpn$ be a domain. A \textit{Herglotz vector field of order $d\in [1,\infty]$} on D is a mapping $G: D\times \mathbb{R}_+ \rightarrow \Cpn$ with the following properties:
     \begin{enumerate}
         \item[(i)] 
       The mapping $G(z,\cdot)$ is measurable on $\mathbb{R}^+$ for all $z\in D$.
       \item[(ii)]  The mapping $G(\cdot,t)$ is holomorphic vector field on $D$ for all $t \in \mathbb{R}^+$.
       \item[(iii)]  For any compact set $K\subset D$ and all $T>0$, there exists a function $C_{T,K} \in L^d([0,T],\mathbb{R}_+)$ such that 
        \begin{equation*}
            ||G(z,t)||\leq C_{T,K}(t), \ \ \ z\in K, \ a.e. \ t \in[0,T].
        \end{equation*}
    \item[(iv)]   $D\in z\rightarrow G(z,t)$ is semicomplete for almost every $t\in [0,+\infty)$.
 \end{enumerate}
\end{definition}
\begin{definition}\label{Definition-evolution family}
    Let $D\subset \Cpn$ be a domain. A family $(\phi_{s,t})_{0\leq s\leq t}$ of holomorphic self-mappings of $D$ is an \textit{evolution family of order $d\in[1,+\infty]$} if it satisfies
    \begin{equation*}
        \phi_{s,s}=id, \ \ \phi_{s,t}=\phi_{u,t}\circ\phi_{s,u}, \ \ 0\leq s\leq u\leq t, 
    \end{equation*}
    and if for any $T>0$ and for any compact set $K\subset\subset D$ there exists a function $c_{T,K}\in  L^d([0,T],\mathbb{R}_+)$ such that
    \begin{equation}
        ||\phi_{s,t}(z)-\phi_{s,u}(z)||\leq \int_u^tc_{T,K}(\xi)d\xi, \ \ z\in K, \ 0\leq s\leq u\leq t\leq T.
    \end{equation}
\end{definition}
The following result is required in our proof of Theorem \ref{Solution of Loewner on holomorphically convex domains in density property}. This shows an one-one correspondence between the Herglotz vector fields of order $d$ and evolution family of order $d$. It is taken from \cite{ArosioBracci}(see also \cite{BracciContrerasMadrigal2}).
\begin{result}\cite[Theorem~1.5]{ArosioBracci} \label{Starlike-existence of phi}
    Let $D$ be a complete hyperbolic domain in $\Cpn$. Then for any Herglotz vector field $G(z,t)$ of order $d\in [1,\infty]$ there exists a unique evolution family $(\phi_{s,t})$ of order $d$ over $D$ such that for all $z\in D$
    \begin{equation}\label{Loewner ODE}
        \frac{\partial\phi_{s,t}}{\partial t}(z)=G(\phi_{s,t}(z),t) \ \ a.e. \ t\in [s,+\infty).
    \end{equation}
    Conversely, for any evolution family $(\phi_{s,t})$ of order $d\in [1,\infty]$ over $D$ there exists a Herglotz vector field $G$ of order $d$ such that (\ref{Loewner ODE}) is satisfied. Moreover, if $H$ is another weak holomorphic vector field which satisfies (\ref{Loewner ODE}) then $G(z,t)=H(z,t)$ for all $z\in D$ and almost every $t\in \mathbb{R^{+}}$.
\end{result}
We now define the Loewner PDE and its solutions.
\begin{definition}
        Let $D$ be a complete hyperbolic domain in $\Cpn$. The partial differential equation 
    \begin{equation}\label{Loewner PDE}
      \frac{\partial f_t}{\partial t}(z)=-df_t(z)G(z,t), \ a.e \ t\geq 0,\ z\in D, 
     \end{equation}
     where $G(z,t)$ is a Herglotz vector field of order $d\in [1,+\infty]$, is called the Loewner PDE. A $solution$ to (\ref{Loewner PDE}) is a family of holomorphic mappings $(f_t)$ from $D$ to a complex manifold $Q$ of dimension $n$ such that the following holds:
     \begin{itemize}
  \item[(i)]  The mapping $t\rightarrow f_t$ is continuous with respect to the compact open topology in $Hol(D,Q)$.
  \item [(ii)] For all fixed $z\in D$ the mapping $t\rightarrow f_t(z)$ is locally absolutely continuous in $\mathbb{R}_+$.
   \item [(iii)] For all fixed $z\in D$ equality (\ref{Loewner PDE}) holds almost everywhere in $\mathbb{R}_+$.
     \end{itemize}
     \end{definition}
The following proposition is mentioned in \cite{ArosioBracciWold}. It is the generalized version of Proposition $2.3$ in  \cite{ContrerasMadrigalGumenyuk}. It also follows from different results of \cite{ArosioBracciHamadaKohr}.
     \begin{result}\cite[Proposition~2.6]{ArosioBracciWold}\label{Properties of solutions of Loewner chain}
         Let $D\subset \Cpn$ be a complete hyperbolic domain, let $Q$ be a complex manifold of dimension $n$ endowed with a Hermitian distance $d_Q$. Let $G(z,t)$ be a Herglotz vector field on $D$ of order $d\in [1,+\infty]$. Let $(f_t:D\rightarrow Q)$ be a solution to the Loewner PDE (\ref{Loewner PDE}).  Then 
         \begin{itemize}
         \item [(i)] The family $(f_t:D\rightarrow Q)$ is of order $d$, that is for any compact set $K\subset D$ and all $T>0$, there exists a function $c_{K,T}\in L^d([0,T],\mathbb{R}_+)$ such that 
         \[
         d_Q(f_s(z),f_t(z))\leq \int^t_s C_{K,T}(\zeta)d\zeta, \ \ z\in K, 0\leq s \leq t\leq T.
         \]
         
         \item[(ii)] There exists a set of zero measure $E\subset \mathbb{R}_+$ such that for all $z\in D$ and all $t\in \mathbb{R}_+\setminus E$ the partial derivative $\partial f_t(z)/\partial t$ exists and equality \ref{Loewner PDE} holds.
         \smallskip
         
         \item[(iii)] If $(\phi_{s,t}:D\rightarrow D)$ is the evolution family which solves the Loewner ODE (\ref{Loewner ODE}), then 
         \[
         f_s=f_t\circ \phi_{s,t}, \ \ 0\leq s \leq t.
         \]
         \end{itemize}
     \end{result}
 The following theorem is extremely important for our work as it has  reduced the problem of solving the Loewner PDE in a complete hyperbolic domain in $\Cpn$ to the embedding problem
of $n$-dimensional Loewner range into $\Cpn$. The following result is written with various results of \cite{ArosioBracci}, \cite{ArosioBracciHamadaKohr}  and it is mentioned in \cite{ArosioBracciWold}.
\begin{result}\cite[Theorem~2.8]{ArosioBracciWold}\label{Starlike-existence of univalent family for Loewner PDE}
Let $D\subset \Cpn$ be a complete hyperbolic domain and $G(z,t)$  be a Herglotz vector field of order $d\in [1,\infty]$ on $D$. Let $(\phi_{s,t}:D\rightarrow D)$ be the evolution family (necessarily of order $d$) which solves the Loewner ODE (\ref{Loewner ODE}).\\ If $Q$ is a complex manifold of dimension $n$ and $(f_t:D\rightarrow Q)$ is a family of univalent mappings satisfying the functional equation  $f_s=f_t\circ\phi_{s,t}$ $(0\leq s\leq t)$ then $(f_t)$ is a solution of order $d$ to the Loewner PDE (\ref{Loewner PDE}). \\ Furthermore there always exists a family of univalent mappings $(f_t:D\rightarrow N)$ satisfying the functional equation   $f_s=f_t\circ\phi_{s,t}$ $(0\leq s\leq t)$ with values in an $n$-dimensional complex manifold $N$ which depends on $G(z,t)$.
\end{result}
\noindent   The manifold $\bigcup_{t\geq 0}f_t(D)$ is called the Loewner Range associated to the evolution family $(\phi_{s,t})$ $(\text{or the Herglotz vector field } G(z,t))$.

 \begin{result}\cite[Corollary~5.4]{ArosioBracciHamadaKohr}\label{Solution of Loewner PDE with values in a manifold}
         Let $D$ be a complete hyperbolic domain in $\Cpn$. Let $G:D\times \mathbb{R}_+\rightarrow\Cpn$ be a Herglotz vector field of order $d\in [1,+\infty]$. Then there exists a Loewner chain $(f_t: D\rightarrow N)$ of order $d$ which solves the Loewner PDE
          \begin{equation}\label{Loewner PDE 2}
      \frac{\partial f_t}{\partial t}(z)=-df_t(z)G(z,t),\ \ a.e. \  t\geq 0, \ \forall z\in D,  
     \end{equation} where $N=\bigcup_{t\geq 0}f_t(D)$ is a complex manifold of dimension $n$ and any other solution to (\ref{Loewner PDE 2}) with values in a complex manifold $Q$ is of the form $(\Lambda\circ f_t)$ where $\Lambda : N\rightarrow Q$ is holomorphic.
     \end{result}
   
We will use the next result in the proof of Theorem \ref{Solution of Loewner on holomorphically convex domains in density property}. Since every complete hyperbolic manifold is taut (see \cite{Kiernan}), the following result is also valid for any complete hyperbolic manifold. It follows from the definition of evolution family (see Definition \ref{Definition-evolution family}) that any evolution family of order $\infty$ is an evolution family of order $d=1$. Hence the following result is also true for evolution family of order $\infty$.
\begin{result}\cite[Lemma ~2]{BracciContrerasMadrigal2}\label{continuity of phis,t}
    Let $M$ be a taut complex manifold. Let $(\phi_{s,t})$ be an evolution family of order $d\geq1$ in $M$. The map $(s,t)\rightarrow \phi_{s,t}\in Hol(M,M)$ is jointly continuous. Namely, given a compact set $K\subset M$ and two sequences $\{s_n\},\{t_n\}$ in $[0,+\infty)$, with $0\leq s_n \leq t_n, \ s_n \rightarrow s, \ t_n \rightarrow t$, then $\lim_{n\rightarrow \infty}\phi_{s_n,t_n}=\phi_{s,t}$ uniformly on $K$.
\end{result}

The following result is shown by H. Hamada in \cite{Hamada}. We will use this result in our proof of Theorem~\ref{Solution of Loewner on holomorphically convex domains in density property}.
\begin{result}\cite[Theorem 5.1]{Hamada}\label{Runge pairs in Loewner range Hamada}
    Let $M$ be a complete hyperbolic Stein manifold and $Q$ be a complex manifold of the same dimension. Let $(f_t:M\rightarrow Q)$ be a Loewner chain of order $d\in [1,\infty]$. Then $(f_{s_1}(M),f_{s_2}(M))$ is a Runge pair for $0\leq s_1 <s_2$. 
    \end{result}
       Note that, a Solution of the Loewner PDE (\ref{Loewner PDE}) for a Herglotz vector field of order $d\in [1,\infty]$, is a Loewner chain of order $d$  \cite[Definition~2.8]{Hamada}.  This follows from Result \ref{Properties of solutions of Loewner chain}.


\section{Approximations and Runge embeddings }\label{sec-main}
In this section we will prove Theorem~\ref{Main theorem for embedding Stein manifolds}.

\subsection{Approximation of holomorphic embeddings of same dimension}

\begin{proposition}
   \label{Continuous isotopy and approximation of Y valued biholomorphisms}
    Let $\Omega$ and $\Ot$ be two non-empty Stein domains in a complex manifold $M$ of dimension $n$, $n\geq2$, such that $\Omega\subset\Ot$. Let $H:[0,1]\times \Omega \rightarrow \Omega $ be a continuous isotopy of injective holomorphic maps from $\Omega$ into $\Omega$ such that 
    \begin{itemize}
   \item[(i)] $H(0,z)=z$ for all $z\in\Omega$; 
   \item [(ii)] $H_1 :\Omega\rightarrow\Omega$ extends to an injective holomorphic map $\tilde{H_1}: \tilde{\Omega}\rightarrow\Omega$; and
   \item[(iii)]$(H_t(\Omega),\Omega)$, $(\tilde{H_1}(\Ot),\Omega)$ are Runge pairs in $M$ for all $t\in [0,1]$.
   \end{itemize}
    Then, for any Stein manifold $W$ of dimension $n$ with the density property and any holomorphic embedding $F: \Omega \rightarrow W$ such that $(F(\Omega), W)$ is a Runge pair, can be approximated uniformly on compacts of $\Omega$ by holomorphic embeddings of $\Ot$ into $W$ whose images are Runge in $W$. In other words, $\mathcal{E}_{\mathcal{R}}(\Omega,W)\subset \overline{\mathcal{E}_{\mathcal{R}}(\tilde{\Omega},W)}|_{\Omega}$ in $\mathcal{E}(\Omega,W)$.
 \end{proposition}
\begin{proof}
 If $\mathcal{E}_{\mathcal{R}}(\Omega,W)=\emptyset$ then the conclusion follows trivially. Assume $\mathcal{E}_{\mathcal{R}}(\Omega,W)\neq \emptyset$ and $F\in \mathcal{E}_{\mathcal{R}}(\Omega,W)$. 
  Let us now construct a new isotopy from the given $H$ by,  
  \begin{equation*}
   G:[0,1]\times F(\Omega)\rightarrow W , \ \    G(t,w)=F(H(t,F^{-1}(w))).
    \end{equation*}
    $G$ is clearly continuous and each $G_t=G(t,\cdot):F(\Omega)\rightarrow W$ is an injective holomorphic on $F(\Omega)$.
    Now on $\Omega$ we have,
    \begin{equation}\label{F-related1}
    G_t\circ F=F\circ H_t , \ \  t\in [0,1].
     \end{equation}
We first show that $G_t(F(\Omega))$ is Runge in $W$ for each $t \in [0,1]$. It is equivalent to show $F(H_t(\Omega))$ is Runge in $W$ by (\ref{F-related1}).
Since $(H_t(\Omega),\Omega)$ is a Runge pair in $M$ for each $t\in [0,1]$. Hence, we obtain that $(F(H_t(\Omega)),F(\Omega))$ is a Runge pair in $W$.
Since $F(\Omega)$ is Runge domain in $W$, we get $F(H_t(\Omega))$ is Runge in $W$.
Now, we have a continuous isotopy $G_t$ on $F(\Omega)\subset W$ that satisfies $G_0(w)=w, \ \forall w\in F(\Omega)$, and each $G_t(F(\Omega))$ is Runge domain in $W$. Using  Lemma \ref{Continuous isotopy in Manifold with the density-modified}, we get a sequence $\{\Psi_k\}_{k\in \mathbb{N}}$ of holomorphic automorphisms of $W$ that converges to $G^{-1}_1$ locally uniformly on $G_1(F(\Omega))$.
Take a sequence of holomorphic maps from $F_k:\tilde{\Omega}\rightarrow W $, $k\in \mathbb{N},$ defined by, $$F_k=\Psi_k\circ F\circ \tilde{H}_1 .$$
Each $F_k$ is well defined as $\tilde{H}_1(\Ot)\subset \Omega$ , and is an injective holomorphic map from $\Ot$ into $W$.
On $\Omega$ we have, $$F_k|_{\Omega}=\Psi_k\circ F\circ H_1=\Psi_k\circ G_1\circ F.$$
The sequence $\Psi_k$ converges to $G^{-1}_1$ locally uniformly on $G_1(F(\Omega))$. This implies $\Psi_k\circ G_1$ converges to $Id_{F(\Omega)}$ locally uniformly on $F(\Omega)$. Hence $\Psi_k\circ G_1\circ F=F_k|_{\Omega}$ converges to $F$ locally uniformly on $\Omega$.

Now we show for each $k\in \mathbb{N}$, $F_k(\Ot)$ is a Runge domain in $W$. 
Since $(\tilde{H}_1(\Ot),\Omega)$ is a Runge pair in $M$, this gives $(F(\tilde{H}_1(\Ot)),F(\Omega))$ is a Runge pair. Since $F(\Omega)$ is Runge in $W$, $F(\tilde{H}_1(\tilde{\Omega}))$ is a Runge domain in $W$.
So that $\Psi_k(F(\tilde{H}_1(\Ot)))=F_k(\Ot)$ is a Runge domain in $W$ for each $k\in \natu$ as image of a Runge domain  under an automorphism of $W$. Therefore we get a sequence $\{F_k\}_{k\in \mathbb{N}}$ of holomorphic embeddings from $\tilde{\Omega}$ into $W$ that approximates $F$ uniformly on compacts of $\Omega$ and each $F_k(\tilde{\Omega})$ is Runge in $W$. In other words, the closure of $\mathcal{E}_{\mathcal{R}}(\tilde{\Omega},W)|_{\Omega}$ in $\mathcal{E}(\Omega,W)$ contains $\mathcal{E}_{\mathcal{R}}(\Omega,W)$.

\end{proof}
\begin{remark}
    It is not required to assume  $\Omega$, $\tilde{\Omega}$ to be Stein domains in Proposition~\ref{Continuous isotopy and approximation of Y valued biholomorphisms} if $W=\cplx^n$. The proof is exactly same but we need to use Result \ref{Forstneric Rosay continuous isotopy} in the proof in place of Result \ref{Continuous isotopy in Manifold with the density}.
\end{remark}

\subsection{Embedding of the increasing union}
We now prove a theorem to extend the process of approximation for an increasing sequence of domains in a complex manifold.  
 

\begin{theorem}\label{Union embedding in domain with Density}
    Let $X$ and $W$ be two Stein manifolds of the same dimension. Let $\{\Omega_j\}_{j\in \mathbb{N}}$ be a family of connected non-empty open subsets of $X$ such that $\Omega_1\subset \hdots \subset \Omega_j\subset \hdots \subset\cup_{j\in \mathbb{N}}\Omega_j=X$. Assume that $\mathcal{E}_{\mathcal{R}}(\Omega_1,W)\neq \emptyset$ and $\mathcal{E}_{\mathcal{R}}(\Omega_j,W)\subset\overline{{\mathcal{E}}_{\mathcal{R}}(\Omega_{j+1},W)}|_{\Omega_j}$ in $\mathcal{E}(\Omega_j,W)$ for each $j\in \mathbb{N}$, then $X$ admits a holomorphic Stein and Runge embedding in $W$.
    Furthermore, for any $j\in \mathbb{N}$, $\mathcal{E}_{\mathcal{R}}(\Omega_j,W)\subset\overline{{\mathcal{E}}_{\mathcal{R}}(X,W)}|_{\Omega_j}$ in $\mathcal{E}(\Omega_j,W)$.
\end{theorem}

\begin{proof}
    Since $X$ is a Stein manifold there exists a compact exhaustion $\{K_j\}_{j\in \mathbb{N}}$ of $X$ such that for each $j\in \mathbb{N}$ the compact $K_j$ is $\mathcal{O}(X)$-convex, $K_1\neq \emptyset$ and $K_j\subset K_{j+1}^{\circ}$. Fix $j\in \mathbb{N}$. We check that there exists a $\psi(j)\in \mathbb{N} $ such that $K_j\subset \Omega_{\psi(j)}$. Let us assume that such $\psi(j)$ does not exist. Then there exists a sequence $\{x_m\}_{m\in \mathbb{N}}$ in $K_j$  such that $x_m\in \Omega_{m+1}\setminus\Omega_m$, for each $m\in\mathbb{N}$. Since $K_j$ is compact, it has a subsequence $\{x_{m_k}\}$ that converges to $x$ in $K_j$. Now  $x\in \Omega_{j_0}$ for some $j_0\in \mathbb{N}$. This implies that $x_{m_k}\in \Omega_{j_0}, \ \forall k\geq k_0$ for some $k_0\in \mathbb{N}.$ This contradicts the fact $x_{m_k}\in \Omega_{m_k+1}\setminus\Omega_{m_k}$ as $x_{m_k}\notin \Omega_{j_0},  \forall m_k\geq j_0$. So that there exists a $\psi(j)\in \mathbb{N}$ such that $K_j\subset\Omega_{\psi(j)}$.   We can choose $\psi:\mathbb{N}\rightarrow \mathbb{N}$ such that $\psi(j)<\psi(j+1).$ Also from the assumption it follows that $\mathcal{E}_{\mathcal{R}}(\Omega_{\psi(j)},W)\subset\overline{{\mathcal{E}}_{\mathcal{R}}(\Omega_{\psi(j+1)},W)}|_{\Omega_j}$ in $\mathcal{E}(\Omega_{\psi(j)},W)$. Therefore, without loss of generality, we may assume $K_j\subset \Omega_j$ for each $j\in \mathbb{N}.$
    
    We first construct a sequence of injective holomorphic maps that converges uniformly on each compact of $X$ to an injective holomorphic map from $X$ into $W$.
    Since $\mathcal{E}_{\mathcal{R}}(\Omega_1,W)\neq \emptyset$, let $\Phi_1:\Omega_1\rightarrow W$ be an injective holomorphic map such that $(\Phi_1(\Omega_1),W)$ is a Runge pair. Again   $\mathcal{E}_{\mathcal{R}}(\Omega_1,W)\subset\overline{{\mathcal{E}}_{\mathcal{R}}(\Omega_{2},W)}|_{\Omega_1}$ and $K_1\subset \Omega_1$. This implies that given any $\epsilon >0$, we can get a $\Phi_2: \Omega_2 \rightarrow W$ such that $(\Phi_2(\Omega_2),W)$ is a Runge pair in $W$ and  $$\underset{x\in K_1}{sup}d_W(\Phi_2(x),\Phi_1(x))<\frac{\epsilon}{2^2},$$ 
    where $d_W:W\times W\rightarrow \mathbb{R}_+$ be the distance function that induces the topology on $W$. Continuing the procedure we obtain a sequence $\{\Phi_j\}_{j\in \mathbb{N}}$  $(\Phi_j: \Omega_j \rightarrow W)$ such that $(\Phi_{j}(\Omega_{j}),W)$ is Runge pair in $W\subseteq\Cpn$ and it satisfies that $$\underset{x\in K_j}{sup}d_W(\Phi_{j+1}(x),\Phi_j(x))<\frac{\epsilon}{2^{j+1}}.$$
 Note that, the sequence $\{\Phi_j\}_{j\in \mathbb{N}}$ is uniformly Cauchy in each compact of $X$. So that it converges uniformly on each compact set of $X$ to a holomorphic map $\Phi: X \rightarrow W$.\\
  If we choose $\epsilon$ small enough, then the limit $\Phi$ will be non-degenerate on $X$. Since $\Phi$ is non-degenerate and limit of injective holomorphic maps, using Lemma \ref{Hurwitz-non-degenerate limit of injective} it follows that it is itself an injective holomorphic map from $X$ into $W$.\\
Next, we claim that $\Phi(X)$ is Runge domain in $W$. Let $j\geq m$.
 Since $K_m$ is $\mathcal{O}(X)$-convex, by Oka-Weil theorem we get that $\mathcal{O}(X)|_{K_m}$ is dense in $\mathcal{O}(K_m)$. Again by assumption we know that \[K_m\subset \Omega_m \subset \Omega_j\subset X.\]
 It implies that $\mathcal{O}(X)|_{K_m}\subset \mathcal{O}(\Omega_j)|_{K_m}\subset \mathcal{O}(K_m)$. Therefore we get that $\mathcal{O}(\Omega_j)|_{K_m}$ is dense in $\mathcal{O}(K_m)$, $\forall j\geq m$.
Since $\Phi_j:\Omega_j\rightarrow W$ is an injective holomorphic map and  $\Phi_j(K_m)\subset\Phi_j(\Omega_j)$, by Lemma \ref{Rungepair biholomorphic invariant}, it follows that $\mathcal{O}(\Phi_j(\Omega_j))|_{\Phi_j(K_m)}$ is dense $\mathcal{O}(\Phi_j(K_m))$. So that, for any $f\in \mathcal{O}(\Phi_j(K_m))$ there exists a sequence $\{f_k\}_{k\in \mathbb{N}}$ in $\mathcal{O}(\Phi_j(\Omega_j))$ that converges to $f$ as $k$ tends to $\infty$, uniformly on $\Phi_j(K_m)$. Again we know that $(\Phi_j(\Omega_j),W)$ is Runge pair in $W$. Hence, for each $f_k\in \mathcal{O}(\Phi_j(\Omega_j))$ there exists a sequence $\{g_{k,i}\}_{i\in \mathbb{N}}$ in $\mathcal{O}(W)$ such that $g_{k,i}$ converges to $f_k$ as $i$ tends to $\infty$, uniformly on $\Phi_j(K_m)$. This implies that $g_{k,k}$ converges to $f$ as $k$ tends to $\infty$, uniformly on $\Phi_j(K_m)$. Therefore,
 for each $j,m\in \mathbb{N}$ with $j\geq m$, we conclude that \[\mathcal{O}(W)|_{\Phi_j(K_m)} \text{ is dense }\mathcal{O}(\Phi_j(K_m)).\]
 Next we show that for each $N$, there exists an integer $j(N)\in \mathbb{N}$ such that $$\Phi(K_N)\subset \Phi_{j(N)}(K_{N+1})\subset \Phi(X).$$

For $j>N$, $\Phi_j\circ\Phi^{-1}:\Phi(K^{\circ}_{N+1})\rightarrow W$ is a sequence of injective holomorphic maps converging to $Id$ on $\Phi(K^{\circ}_{N+1})$. For the compact set $\Phi(K_N)\subset Id(\Phi(K^{\circ}_{N+1}))$, by Lemma \ref{eventual containment of compacts for manifolds} there exists $m_1\in \mathbb{N}$ such that $$\Phi(K_N)\subset \Phi_j\circ\Phi^{-1}(\Phi(K^{\circ}_{N+1})) \ \forall j\geq m_1.$$
 \[ \Rightarrow \Phi(K_N)\subset \Phi_j(K^{\circ}_{N+1})\ \forall j\geq m_1.\]
 By construction, $\Phi$ is non-degenerate at every point of $X$. It implies that $\Phi(X)$ is open in $W$. Since $\Phi(K_{N+1})$ is compact subset of $\Phi(X)$, take an open subset $U$  in $\Phi(X)$ such that $\Phi(K_{N+1})\subset U\subset \Phi(X)$. Since $\Phi_j$ converges to $\Phi$ uniformly on $K_{N+1}$, so that there exists $m_2$ such that $\Phi_j(K_{N+1})\subset U$, $\forall j\geq m_2$.
Take $j(N)=max\{m_1,m_2,N+1\}$. Then we get that $$\Phi(K_N)\subset \Phi_{j(N)}(K_{N+1})\subset \Phi(X).$$ For each $N\in \mathbb{N}$, let us denote $L_N=\Phi_{j(N)}(K_{N+1})$. Then $L_N\subset \Phi(X)$, and for any compact set $L\subset\Phi(X)$, there exists a $l\in \mathbb{N}$ such that \[L\subset \Phi(K_l)\subset L_l=\Phi_{j_{(l)}}(K_{l+1})\subset \Phi(X).\] Since $\mathcal{O}(W)|_{L_l}$ is dense in $\mathcal{O}(L_l)$,  any $g\in \mathcal{O}(\Phi(X))$ can be approximated uniformly on $L\subset L_l$ by elements of $\mathcal{O}(W)$. Since $L$ is any compact set in $\Phi(X)$, we get that $(\Phi(X),W)$ is a Runge pair.
Therefore, we get  an injective holomorphic map $\Phi :X \rightarrow W$ such that $(\Phi(X),W)$ is a Runge pair. Clearly, $\Phi(X)$ is Stein as $X$ and $\Phi(X)$ is biholomorphic to each other.\\

 For the second part, fix $s\in \mathbb{N}$ and take any $F\in\mathcal{E}_{\mathcal{R}}(\Omega_s,W)$. Let $K$ be  any compact set in $\Omega_s$. Then there exists $q\in \mathbb{N}(>s)$ such that $K\subset K_j$ for all $j\geq q >s$.
 Again we have $\mathcal{E}_{\mathcal{R}}(\Omega_s,W)\subset\overline{{\mathcal{E}}_{\mathcal{R}}(\Omega_{q},W)}|_{\Omega_s}$ in $\mathcal{E}(\Omega_s,W)$. Given $\epsilon>0$ there exists a $F_{q}\in\mathcal{E}_{\mathcal{R}}(\Omega_q,W) $ such that \[\underset{x\in K}{sup}d_W(F_q(x),F(x))<\epsilon/2^{q+1}.\]
 We can choose a sequence $\{F_j\}_{j\geq q}$ $(F_j\in \mathcal{E}_{\mathcal{R}}(\Omega_j,W))$ such that $$\underset{x\in K_j}{sup}d_W(F_{j+1}(x),F_j(x))<\epsilon/2^{j+2}, \ \forall j\geq q.$$ Here $\{F_j\}_{j\geq q}$ is uniformly Cauchy in compacts of $X$. Hence $F_j$ converges to  a holomorphic map $\tilde{F}:X \rightarrow W$ locally uniformly on $X$. If we choose $\epsilon$ sufficiently small then $\tilde{F} $ is non degenerate and so injective from $X$ into $W$.
 In the same way as we did in the first part, we have $\tilde{F}(X)\subseteq W$ and $\tilde{F}(X)$ is Runge in $W$. We also obtain that for any $j\geq 0$,\begin{align*}
     \underset{x\in K}{sup}d_W(F_{j+q}(x),F(x))\ &<\underset{x\in K}{sup}d_W(F_{j+q}(x),F_{j-1+q}(x))+\dots+\underset{x\in K}{sup}d_W(F_{q}(x),F(x))\ \\
                    &<\frac{\epsilon}{2^{j+q+1}}+\dots+\frac{\epsilon}{2^{q+1}}<\frac{\epsilon}{2^{q}}.
 \end{align*}
 As $j$ tends to $\infty$, we get that $\underset{x\in K}{sup}d_W(\tilde{F}(x),F(x))<\epsilon$.
 As it is true for any compact $K\subset \Omega_s$ and arbitrarily small $\epsilon>0$, we get that $F\in \overline{{\mathcal{E}}_{\mathcal{R}}(X,W)}|_{\Omega_s}$.
\end{proof}

Now we can give a proof of our main theorem.

\begin{proof}[Proof of Theorem \ref{Main theorem for embedding Stein manifolds}]
    For any $j\in \mathbb{N}$, the pair $(\Omega_j,\Omega_{j+1})$ satisfies assumptions $(i),(ii),$ $(iii)$. So that we can use Proposition \ref{Continuous isotopy and approximation of Y valued biholomorphisms} to conclude that $\overline{\mathcal{E}_{\mathcal{R}}(\Omega_{j+1},W)}|_{\Omega_j}$ in $\mathcal{E}(\Omega_j,W)$ contains $\mathcal{E}_{\mathcal{R}}(\Omega_{j},W)$. By assumption, $\Omega_1\neq \emptyset$ and $\mathcal{E}_{\mathcal{R}}(\Omega_{1},W)\neq \emptyset$. Using Theorem \ref{Union embedding in domain with Density} it follows that $X$ admits a Stein and Runge embedding in $W$, and $\overline{\mathcal{E}_{\mathcal{R}}(X,W)}|_{\Omega_s}$ in $\mathcal{E}(\Omega_s,W)$ contains $\mathcal{E}_{\mathcal{R}}(\Omega_s,W)$ for any $s\in \mathbb{N}$.
    \end{proof}

\subsection{ Runge pairs generated by $(\mathbb{R},+)$-actions on Stein manifolds}
We use the following theorem to prove Theorem~\ref{cor-Approximation on action invariant domains of Cn}. It is helpful for us to bypass the Runge pair condition in Theorem~\ref{Main theorem for embedding Stein manifolds} in case of $(\mathbb{R},+)$ action on Stein manifolds.

\begin{theorem}\label{Checking Runge pairs for R-actions}
    Let $M$ be a complex manifold. $\theta:\mathbb{R}\times M\rightarrow M$ be a continuous group action of $(\mathbb{R},+)$ such that each $\theta_t=\theta(t,\cdot):M\rightarrow M\ (t\in \mathbb{R})$ is holomorphic. Assume that $\Omega$ is a Stein domain which is invariant under each $\theta _t$ for $t\geq 0 \ (\theta_t(z)\in \Omega$ if $z\in \Omega)$. Then $(\theta_{s_2}(\Omega),\theta_{s_1}(\Omega))$ is a Runge pair for $0\leq s_1\leq s_2$. 
\end{theorem}
\begin{proof}
     If $s_1=s_2$ then it follows trivially. Assume that $0\leq s_1<s_2$. We check that $\theta_{s_2}(\Omega)$ is semicontinuously holomorphically extendable  to $\theta_{s_1}(\Omega)$ by checking conditions $(1)-(5)$ in Definition \ref{definition-semicontinuously holomorphically extendable}. 
    Each $\theta_t(\Omega) \ (t\in \mathbb{R})$ is a Stein domain in $M$ as $\Omega$ is Stein and $\theta_t\in Aut(M)$. This checks condition $(1)$. 
   Let $w\in \theta_t(\Omega).$ Then, for any $s<t$ ,\[w\in \theta_s\circ\theta_{t-s}(\Omega)\subset \theta_s(\Omega), \text{ as $\theta_{t-s}(\Omega)\subset\Omega$}\Rightarrow \theta_t(\Omega)\subset\theta_s(\Omega)\] Therefore, we get that 
   \[
   \theta_t(\Omega)\subset \theta_s(\Omega) \text{ when $0<s<t$ . This checks condition $(2)$.}
   \]
     We now show that $\bigcup_{s< t\leq s_2}\theta_t(\Omega)=\theta_s(\Omega)$ for any $s\in [s_1,s_2)$.
      \[\text{ Since } \theta_t(\Omega)\subset \theta_s(\Omega), \ s<t\leq s_2 \Rightarrow \bigcup_{s< t\leq s_2}\theta_t(\Omega)\subset\theta_s(\Omega).\]
      
Using the continuity of $\theta$, we obtain that $\theta_{t_k}$ converges to $\theta_s$ locally uniformly on $\Omega$, as $t_k$ converges to $s$ for any sequence $\{t_k\}_{k\in \mathbb{N}}$ in $(s,s_2]$. Let $w\in \theta_s(\Omega)$. Since each $\theta_{t_k},\theta_s$ is injective holomorphic map, we use Lemma \ref{eventual containment of compacts for manifolds} to conclude that $w\in \theta_{t_k}(\Omega), \ \forall k\geq N_1$ for some $N_1\in \mathbb{N}$. So that $w\in \bigcup_{s< t\leq s_2}\theta_t(\Omega)$ and hence $\bigcup_{s< t\leq s_2}\theta_t(\Omega)=\theta_s(\Omega)$. This checks condition $(4)$. If we choose $s=s_1$ then  $\bigcup_{s_1\leq t\leq s_2}\theta_t(\Omega)=\theta_{s_1}(\Omega)$. This gives Condition $(2)$. We now check that Condition $(5)$ holds.
Let $U$ be the union of connected components of interior of $\bigcap_{s_1\leq t<s}\theta_t(\Omega)$. Then $U\subset \bigcap_{s_1\leq t<s}\theta_t(\Omega)$. Again we know that $\theta_s(\Omega)\subset \theta_t(\Omega)$ for any $t<s$, so that $\theta_s(\Omega)\subset \bigcap_{s_1\leq t<s} \theta_t(\Omega) $. Since $\theta_s(\Omega)$ is open, we have $\theta_s(\Omega)\subset U$. Let $b\in U$. Let us 
take a sequence $\{t_j\}_{j\in \mathbb{N}}$ in $[s_1,s).$
Since $\theta_{-t_j}$ converges to $\theta_{-s}$ locally uniformly on $U$ as $t_j$ tends to $s$, and $\{\theta_{-s}(b)\}\subset \theta_{-s}(U)$. Using Lemma \ref{eventual containment of compacts for manifolds}, we can get $\{\theta_{-s}(b)\}\subset \theta_{-t_j}(U) ,\ \forall j\geq j_0$ for some $j_0\in \mathbb{N}$. Now we have $\theta_{-t_j}(U)\subset \Omega$, and this implies  that $\{\theta_{-s}(b)\}\subset \theta_{-t_j}(U)\subset \Omega, \ \forall j\geq j_0$. Therefore, $b\in \theta_s(\Omega)$. This implies $U=\theta_s(\Omega)$ and Condition $(5)$ is satisfied. Hence, all the conditions of Definition \ref{definition-semicontinuously holomorphically extendable} are satisfied and we obtain that $\theta_{s_2}(\Omega)$ is semicontinuously holomorphically extendable to $\theta_{s_1}(\Omega)$. Therefore, using Result \ref{Runge pair by semicontinuous extension}, we conclude that $(\theta_{s_2}(\Omega),\theta_{s_1}(\Omega))$ is a Runge pair.
    
\end{proof}

     
   We now present a proof of Theorem \ref{Embedding of Stein manifold using R-actions}.

\begin{proof}[Proof of Theorem \ref{Embedding of Stein manifold using R-actions}:] If
 $\mathcal{E}_{\mathcal{R}}(\Omega,W)=\emptyset$ then there is nothing to prove. Let us assume that $\Omega\neq \emptyset$ and $\mathcal{E}_{\mathcal{R}}(\Omega,W)\neq \emptyset$. Note that  $\theta_t\in Aut(X) \ (t\in \mathbb{R})$ as $\theta$ is an action of $(\mathbb{R},+)$. Using Theorem \ref{Checking Runge pairs for R-actions}, we obtain that $(\theta_t(\Omega),\Omega)$ is a Runge pair for each $t\in \mathbb{R}_+$. For each $j\in \mathbb{N}$, since $\theta_{-j}:X\rightarrow X$ is an automorphism, we get that $(\theta_{-j}(\theta_t(\Omega)),\theta_{-j}(\Omega))$ is a Runge pair for each $t\in \mathbb{R}_+, j\in \mathbb{N}$. Again for $0\leq s\leq t$, we know that $\theta_{-s}(\Omega)\subset\theta_{-t}(\Omega)$. This implies that $$X'=\bigcup_{t\geq 0}\theta_{-t}(\Omega)=\bigcup_{j\in \mathbb{N}}\theta_{-j}(\Omega).$$
     Clearly, $X'$ is open in $X$. We now prove that $X'$ is Stein. Each $\theta_{-j}(\Omega)$ is Stein, $\theta_{-j}(\Omega)\subset \theta_{-(j+1)}(\Omega)$, and $(\theta_{-j}(\Omega),\theta_{-(j+1)}(\Omega))$ is a Runge pair. By Result \ref{Stein property of increasing union}, we get that $X'$ is Stein.
    Fix $j\in \mathbb{N}_0$. We now show that $\mathcal{E}_{\mathcal{R}}(\theta_{-j}(\Omega),W)\subset\overline{{\mathcal{E}}_{\mathcal{R}}(\theta_{-(j+1)}(\Omega),W)}|_{\theta_{-j}(\Omega)}$ in $\mathcal{E}(\theta_{-j}(\Omega),W)$. Consider the map $$H:[0,1]\times \theta_{-j}(\Omega)\rightarrow \theta_{-j}(\Omega),\ H(t,w)=\theta_t(w).$$ 
Clearly, $H$ is continuous , $H(0,w)=\theta_0(w)=w$, and each $H_t:\theta_{-j}(\Omega)\rightarrow \theta_{-j}(\Omega)$ is an injective holomorphic map when $t\geq0$. We also know that $(\theta_{-j+t}(\Omega),\theta_{-j}(\Omega))$ is a Runge pair. Define $$\tilde{H}_1:\theta_{-(j+1)}(\Omega)\rightarrow\theta_{-j}(\Omega), \ \tilde{H}_1(w)=\theta_1(w).$$
     Then we have\[\tilde{H}_1|_{\theta_{-j}(\Omega)}=H_1 \text{ on } \theta_{-j}(\Omega),\]  and $(\tilde{H}_1(\theta_{-(j+1)}(\Omega)),\theta_{-j}(\Omega))$ is a Runge pair. Hence using Proposition \ref{Continuous isotopy and approximation of Y valued biholomorphisms}, we obtain that 
     $\mathcal{E}_{\mathcal{R}}(\theta_{-j}(\Omega),W)\subset\overline{{\mathcal{E}}_{\mathcal{R}}(\theta_{-(j+1)}(\Omega),W)}|_{\theta_{-j}(\Omega)}$ in $\mathcal{E}(\theta_{-j}(\Omega),W)$ for any $j\in \mathbb{N}\cup \{0\}$. Since $X'$ is Stein, we can use Theorem \ref{Union embedding in domain with Density} to get that  \[\mathcal{E}_{\mathcal{R}}(\theta_{0}(\Omega),W)=\mathcal{E}_{\mathcal{R}}(\Omega,W)\subset\overline{\mathcal{E}_{\mathcal{R}}(X',W)}|_{\Omega} \text{ in }\mathcal{E}(\Omega,W).\] This implies that any biholomorphism $F:\Omega\rightarrow W$ such that $F(\Omega)$ Runge in $W$ can be approximated by a sequence of biholomorphisms $F_k: X' \rightarrow W\ \ (k\in \mathbb{N})$ locally uniformly on $\Omega$, where $X'=\bigcup_{t\geq 0}\theta_{-t}(\Omega)$. 
\end{proof}

\begin{proof}[Proof of Theorem \ref{cor-Approximation on action invariant domains of Cn}:]
    Let $F\in \mathcal{E}_{\mathcal{R}}(\Omega,\Cpn)$. We first show that $F$ can be approximated by elements of  $\mathcal{E}_{\mathcal{R}}(\Cpn,\Cpn)$ locally uniformly on $\Omega$. By assumption, for each $z\in \Cpn$ the orbit $\{\theta_t(z):t\in \mathbb{R}\}$ intersects $\Omega$. So that, for any $z\in \Cpn$ there exists a $T_z\in \mathbb{R}$ such that $\theta_{T_z}(z)\in \Omega$. Since $\Omega$ is invariant under $\theta_t(\forall t\geq 0)$ we get that $\theta_t(z)\in \Omega, \ \forall t\geq T_z $. Therefore $z\in \theta_{-t}(\Omega), \ \forall t\geq T_z$. This implies that $\bigcup_{t\geq 0}\theta_{-t}(\Omega)=\Cpn$. Now we can use Theorem \ref{Embedding of Stein manifold using R-actions} to conclude that   $\mathcal{E}_{\mathcal{R}}(\Omega,\Cpn)\subset\overline{\mathcal{E}_{\mathcal{R}}(\Cpn,\Cpn)}|_{\Omega}$ in $\mathcal{E}(\Omega,\Cpn)$. Again using  Result \ref{Andersen-Lempert Star theorem}, we get that $\mathcal{E}_{\mathcal{R}}(\Cpn,\Cpn)\subset\overline{Aut(\Cpn)}$ in $\mathcal{E}(\Cpn,\Cpn)$ as $\Cpn$ is star-shaped. Therefore, any $F\in \mathcal{E}_{\mathcal{R}}(\Omega,\Cpn)$ can be approximated by automorphisms of $\Cpn$ locally uniformly on $\Omega$.
\end{proof}
  \begin{remark} If we take the action $\theta:\mathbb{R}\times \Cpn\rightarrow \Cpn$ by $\theta(t,z)=e^{-t}z, \ z\in \Cpn, t \in \mathbb{R}$ then any star-shaped domain with center $0$ is invariant under the action $\theta_t (\forall t\geq 0)$.  Result \ref{Andersen-Lempert Star theorem} for star-shaped Stein domains follows from it. This also generalizes other theorems on approximation of biholomorphisms by automorphisms of $\Cpn$ on Stein domains of $\Cpn$, proven in \cite[Theorem~1.1]{Forstneric2025}, \cite[Theorem~4.2]{Hamada}, \cite[Theorem~1.5,V5]{ChatterjeeGorai}. 
 \end{remark}
 
We did not find a proof of the existence of Runge neighborhood basis of variety given by a single entire function in the literature. So we include it. Here $B_r(a)$ denotes the ball of radius $r>0$ in $\Cpn$ centered at $a$, and $D_r$ denotes the disc of radius $r>0$ in $\mathbb{C}$ centered at $0$.
\begin{lemma}\label{Lemma- Runge neighborhood of a variety}
    Let $f\in \mathcal{O}(\Cpn)$. Then for any open set $U$ that contains $S=\{(z,z')\in \Cpn\times\mathbb{C}:f(z)z'-1=0\}$ has a Stein and Runge neighborhood $\Omega$ in $\mathbb{C}^n\times\mathbb{C}$ such that $\Omega\subset U$.
\end{lemma}
\begin{proof}
Let $A=\{z \in \Cpn: f(z)=0\}$. We have $S\subset (\Cpn\setminus A)\times \mathbb{C}$. Without loss of generality we can assume $S\subset U \subset (\Cpn\setminus A)\times \mathbb{C}$. Let us choose the smallest $j_0\in \mathbb{N}$ such that $S\cap (\bar{B}_{j_0}(0)\times \bar{D}_{j_0})\neq \emptyset.$ 
    For each integer $j\geq j_0$ we can choose $r_j>0$ such that $W_j=\{(z,z')\in \Cpn\times \mathbb{C}:||z||\leq j, \ |z'|\leq j,\  |f(z)z'-1|\leq r_j\}\subset U$. Then $S\subset \bigcup_{j\geq j_0}W_j\subset U$. By choosing an approximate convex increasing function $\phi:\mathbb{R}\rightarrow\mathbb{R}_+$ and composing it with the square of the Euclidean norm in $\mathbb{C}^{n+1}$, we can find a $\mathcal{C}^{\infty}$-plurisubharmonic function $\rho:\Cpn\times \mathbb{C}\rightarrow \mathbb{R}_+ $ such that \[\rho(z,z')>2log \frac{1}{r_{j_0}},\ \forall (z,z')\in \bar{B}_{j_0}(0)\times \bar{D}_{j_0},\] \[\rho(z,z')>2log \frac{1}{r_j},\ \forall (z,z')\in (\bar{B}_{j}(0)\times \bar{D}_j)\setminus(B_{j-1}(0)\times D_{j-1}),j\geq j_0+1.\]
    Define $\sigma:\Cpn\times \mathbb{C}\rightarrow \mathbb{R}_+$  by \[\sigma(z,z')=e^{\rho(z,z')}|f(z)z'-1|^2.\]  Let us choose $\Omega=\{(z,z')\in \Cpn\times\mathbb{C}:\sigma(z,z')<1\}$. Since $\sigma$ is a $\mathcal{C}^{\infty}$-plurisubharmonic function it follows that $\Omega$ is Stein and Runge in $\Cpn\times \mathbb{C}$(see \cite[Ch. VI, Theorem~1.18]{MRange}). Let $(a,a')\in \Omega$. If $(a,a')\in (\bar{B}_{j}(0)\times \bar{D}_j)\setminus(B_{j-1}(0)\times D_{j-1})$ for $j\geq j_0+1$ then we get \[ e^{\rho(a,a')}|f(a)a'-1|^2<1\Rightarrow |f(a)a'-1|^2<e^{-\rho(a,a')}\Rightarrow |f(a)a'-1|^2<e^{-2log \frac{1}{r_j}}<r^2_j.  \] The same holds if $(a,a')\in \bar{B}_{j_0}(0)\times \bar{D}_{j_0}$. This implies $(a,a')\in W_j\subset U(j\geq j_0)$ and $S\subset \Omega \subset \cup_{j\geq j_0}W_j\subset U$. Since $\sigma $ is a $\mathcal{C}^{\infty}$-plurisubharmonic function on $\Cpn\times \mathbb{C}$, $\Omega$ is a Stein and Runge domain in $\Cpn\times\mathbb{C}.$
\end{proof}

Now we will give a proof of Theorem \ref{cor-Runge embedding of Cn-A x C}. This proof depends entirely on Theorem \ref{Embedding of Stein manifold using R-actions} and Lemma \ref{Lemma- Runge neighborhood of a variety}.
\begin{proof}[Proof of Theorem \ref{cor-Runge embedding of Cn-A x C}]
    We write $\Cpn\times \mathbb{C}=\{(z,z'):z\in \Cpn,z'\in \mathbb{C}\}$. Let $X=(\Cpn\setminus A )\times \mathbb{C}$. Choose the injective holomorphic map $$\Phi:(\Cpn\setminus A) \times \mathbb{C}\rightarrow \Cpn\times \mathbb{C}, \ \ \Phi(z,z')=(z,\frac{1}{f(z)}+z').$$
    We have $S=\Phi((\Cpn\setminus A)\times \{0\})=\{(w,w')\in \Cpn\times \mathbb{C}:f(w)w'=1\}$. Here $S\subset \Phi(X)$. Using Lemma \ref{Lemma- Runge neighborhood of a variety} we get that  there exists a Stein domain $\Omega'$ such that $(\Omega',\Cpn\times \mathbb{C})$ is a Runge pair and $S\subset\Omega'\subset\Phi(X)$. Also we have $(\Omega',\Phi(X))$ is a Runge pair. Here $\Omega=\Phi^{-1}(\Omega')$ contains the set $(\Cpn\setminus A)\times \{0\}$ and $(\Phi^{-1}(\Omega'),\Phi^{-1}(\Phi(X)))=(\Omega,X)$ is a Runge pair.
    We now construct a Stein domain $D$ such that $(D,X)$ is a Runge pair, and $(\Cpn\setminus A)\times \{0\}\subset D\subset \Omega$. Define $$\rho:\Cpn\setminus A\rightarrow \mathbb{R}, \ \rho(z_1,z_2.\dots,z_n)=|z_1|^2+\dots+|z_n|^2+\frac{1}{|f(z)|^2}.$$ Here $\rho$ is a $\mathcal{C}^{\infty}$ -plurisubharmonic function on $\Cpn\setminus A$ and for $c\in (0,\infty)$ denote  $U_c=\{z\in \Cpn\setminus A: \rho(z)<c+\alpha\}$, where $\alpha=\inf_{z\in \Cpn\setminus A} \rho(z)$. Each $U_c$ is relatively compact in $\Cpn\setminus A$. For each $j\in \mathbb{N}$ there exists a $r_j>0$ such that $\bar{U}_j\times \bar{D}_{r_j}\subset \Omega$. We now choose a monotonically decreasing sequence $\{r_j\}_{j\in \mathbb{N}}$. Clearly, $(\Cpn\setminus A)\times\{0\}\subset D'\subset\Omega$, where $D'=\bigcup_{j\in\mathbb{N}} \bar{U}_j\times \bar{D}_{r_j}$.
    Choosing a convex strictly increasing function $\phi:(0,\infty)\rightarrow(0,\infty)$ approximately we get that $$\phi\circ\rho(z)>2log(\frac{1}{r_j}), \ \text{when} \  z\in \bar{U}_j\setminus U_{j-1}.$$
    Construct the function $\sigma:(\Cpn\setminus A)\times \mathbb{C}\rightarrow \mathbb{R}$, $\sigma(z,z')=|z'|^2e^{\phi(\rho(z))}$ and take the sublevel set $D=\{(z,z')\in (\Cpn\setminus A)\times\mathbb{C}: \sigma(z,z')<1\}$. Here $\sigma$ is a  plurisubharmonic function on $(\Cpn\setminus A)\times \mathbb{C}$. Therefore $D=\{(z,z')\in (\Cpn\setminus A)\times\mathbb{C}:\sigma(z,z')<1\}$ is Stein and Runge in $X=\Cpn\setminus A\times \mathbb{C}\ ($see \cite[Ch. VI, Theorem ~1.18]{MRange}$)$. We now check that $\Cpn\setminus A\times\{0\}\subset D\subset \Omega$. Let $(z,0)\in \Cpn\setminus A\times \{0\} \Rightarrow \sigma(z,0)=0<1\Rightarrow (z,0)\in D.$ Therefore $(\Cpn\setminus A)\times\{0\}\subset D$.
    Let $(z,z')\in D$. Then $(z,0)\in (\Cpn\setminus A)\times \{0\}\subset D$. Since $z\in \bar{U}_j\setminus U_{j-1}$ for some $j\in \mathbb{N}$ we also obtain that
     \[|z'|^2e^{\phi(\rho(z))}<1 
                  \Rightarrow |z'|^2<e^{-\phi(\rho(z))}<e^{-2log(\frac{1}{r_j})}<r^2_j, \ \text{as}\ z\in \bar{U}_j \setminus U_{j-1}.\]
              $   \text{Hence,  }  |z'|<r_j.$ So that $ (z,z')\in \bar{U}_j\times \bar{D}_{r_{j}} .$
    This implies that $D\subset \bigcup _{j\in \mathbb{N}}U_j\times D_{r_j}\subset \Omega$. 
 Now we have $D\subset \Omega \subset X$ and each of $(D,X)$ and $(\Omega,X)$ is a Runge pair. This implies $(D,\Omega)$ is a Runge pair. Therefore $(\Phi(D),\Phi(\Omega))$ is a Runge pair. We also have $\Phi(\Omega)=\Omega'$ is a Runge domain in $\Cpn\times\mathbb{C}$. This implies that $\Phi(D)$ is Runge in $\Cpn\times \mathbb{C}$.\\
    Choose the $\mathbb{R}$-action $\theta:\mathbb{R}\times ((\Cpn\setminus A)\times \mathbb{C})\rightarrow (\Cpn\setminus A)\times \mathbb{C}$ given by $\theta(t,(z,z'))=(z,e^{-t}z')$. It is easy to check $\theta_t(D)\subseteq D$ for $t\geq 0$ and $\bigcup_{t\ge 0}\theta_{-t}(D)=(\Cpn\setminus A)\times \mathbb{C}$. Now we can use Theorem \ref{Embedding of Stein manifold using R-actions} to conclude that there exists a sequence $\Phi_k: (\Cpn\setminus A)\times \mathbb{C}\rightarrow \Cpn\times \mathbb{C}\ (k\in \mathbb{N})$ of injective holomorphic maps into $\Cpn\times \mathbb{C}$ such that each $\Phi_k((\Cpn\setminus A)\times \mathbb{C)}$ is Runge in $\Cpn\times \mathbb{C}$ and $\Phi_k$ converges to $\Phi$ locally uniformly on $D$. This gives our desired result.
\end{proof}
\begin{remark}
    If we take a Stein manifold $X$ such that $X\times \mathbb{C}$ has the density property then it is also true that for any $f\in \mathcal{O}(X)$, $(X\setminus A)\times \mathbb{C}$ admits a Stein and Runge embedding in $X\times \mathbb{C}$, where $A=\{x\in X:f(x)=0\}$. Varolin \cite{Varolin} showed that if $X$ is a Stein complex Lie group then $X\times \mathbb{C}$ has the density property. 
\end{remark}

\begin{example}
Let $\rho:\Cpn\times \mathbb{C}\rightarrow \mathbb{R}$ $(n\geq1)$ be a $\mathcal{C}^2$-smooth function given by $\rho(z',z_{n+1})=\sigma(z')-Rez_{n+1}$, where $\sigma:\Cpn \rightarrow \mathbb{R}$ be a $\mathcal{C}^2$-plurisubharmonic function in $\Cpn$. Choose the domain $\Omega_0=\{(z',z_{n+1})\in \Cpn\times \mathbb{C}:\rho(z',z_{n+1})<0\}\neq \emptyset$. Then any biholomorphism $F:\Omega_0 \rightarrow\mathbb{C}^{n+1}$ with $F(\Omega_0)$ Runge, can be approximated by holomorphic automorphisms of $\mathbb{C}^{n+1}$ locally uniformly on $\Omega$.
\begin{proof}
    Since $\rho: \Cpn \times \mathbb{C}\rightarrow \mathbb{R}$ is a plurisubharmonic function, $\Omega_0$ is a Stein domain. Choose the action $\theta:\mathbb{R}\times (\Cpn\times \mathbb{C})\rightarrow \Cpn\times \mathbb{C}$ given by $\theta(t,(z',z_{n+1}))=(z',z_{n+1}+t)$. For any $(w',w_{n+1})\in \Cpn\times \mathbb{C}$ there exists $t_0$ such that \[\rho(\theta(t_0,(w',w_{n+1})))=\sigma(w')-Re(w_{n+1}+t_0)<0.\]
    Therefore $\theta_{t_0}(w',w_{n+1})\in \Omega_0$, and this implies that each orbit intersects $\Omega_0$. Now we show that $\Omega_0$ is invariant under $\theta_t$. For any $(z',z_{n+1})\in \Omega_0$ and $t\geq 0$ we have the following,
    \begin{align*}
     (z',z_{n+1})\in \Omega_0 &\Leftrightarrow \sigma(z')-Re(z_{n+1})<0 \\
                        &\Leftrightarrow\sigma(z')-Re(z_{n+1}+t)<-t, \ \forall t\in[0,1] \\
                        &\Leftrightarrow \rho(z',z_{n+1}+t)<-t<0 \\
                          &\Leftrightarrow \rho(\theta_t(z',z_{n+1}))<0 \Leftrightarrow  \theta_t(z',z_{n+1})\in \Omega_{0}.               
                             \end{align*} 
                             So that $\Omega_0$ is invariant under $\theta_t \ (t\geq 0)$.
    Using Theorem \ref{cor-Approximation on action invariant domains of Cn}, we get our desired result, \[\mathcal{E}_{\mathcal{O}(\Cpn\times \mathbb{C})}(\Omega_0,\Cpn\times \mathbb{C})\subset \overline{Aut(\Cpn\times \mathbb{C})}|_{\Omega_0}\text{ in } \mathcal{E}(\Omega_0,\Cpn\times \mathbb{C}).\]
\end{proof}

\end{example}
\begin{example}
    Let $\Omega$ be a Stein domain in $\mathbb{C}^n=\mathbb{C}^p\times \mathbb{C}^q \  (z=(z',z'')) \ (p,q\geq1)$ such that $\mathbb{C}^p\times \{0''\}\subset\Omega$ and for each $z\in \mathbb{C}$, $\Omega_z'=\{z''\in \mathbb{C}^q: (z',z'')\in \Omega\}$ is star-shaped with star center $0''\in \mathbb{C}^q$. Then $\mathcal{E}_{\mathcal{R}}(\Omega,\Cpn)=\overline{Aut(\Cpn)}|_{\Omega}$.
     \begin{proof}
     Choose the $(\mathbb{R},+)$ action $\theta:\mathbb{R}\times (\mathbb{C}^p\times \mathbb{C}^q)\rightarrow (\mathbb{C}^p\times \mathbb{C}^q)$ given by $\theta(t,(z',z''))=(z',e^{-t}z'')$. Then each orbit $\{\theta_t(z',z''):t\in \mathbb{R}\}((z',z'')\in \mathbb{C}^p\times \mathbb{C}^q)$ intersects $\Omega$. Now we can use Theorem \ref{cor-Approximation on action invariant domains of Cn} to conclude that $\mathcal{E}_{\mathcal{R}}(\Omega,\Cpn)\subset\overline{Aut(\Cpn)}|_{\Omega}$ in $\mathcal{E}(\Omega,\Cpn)$.
     \end{proof}
\end{example}

\section{Approximation of biholomorphisms by automorphisms of $\Cpn$ and Runge embeddings}\label{sec-approx}
In this section we give a proof of Theorem \ref{Starshaped} and proof its corollaries.  The proof of Theorem \ref{Starshaped} competely depends on Result \ref{Forstneric Rosay continuous isotopy}. We will use the result two times in the proof. For a domain $\Omega\subset \Cpn$ and $s\in \mathbb{C}$, Let \[s\Omega=\{sz\in \Cpn:z\in \Omega\}\] We will use this notation in the following proof.
    \begin{proof}[Proof of Theorem~\ref{Starshaped}]
    Choose $F\in \mathcal{E}(\Omega,\Cpn)$.
    To avoid complications, assume $0$ is the star-center of $D$.
    Also assume that $F(0)=0$ and $DF(0)=I$.
    Let us construct a new homotopy $E: [0,1]\times \Omega\rightarrow \Omega$ from the given $H$ by
    \begin{equation*}
    E(t,z)=\begin{cases}
        H(2t,z) \ \ &if \ 0\leq t\leq 1/2,\\
        2(1-t)H(1,z) \ &if \ 1/2\leq t\leq 1.
        \end{cases}    
    \end{equation*}
    By pasting lemma for continuous maps, $E$ is continuous. For each $t \in [0,1)$, $E_t$ is a biholomorphism from $\Omega$ into $\Omega$.
    We now construct another isotopy of biholomorphisms $\tilde{H}$ using $E$,
    \begin{equation}
        \tilde{H}:[0,1]\times \Omega \rightarrow \Cpn, \ \tilde{H}(t,z)=\begin{cases}
         \frac{F(E(1-t,z))}{t} \ \ &when \ t \neq 0 ,\\
         2H_1(z)  &when \  t=0.
        \end{cases} 
    \end{equation}
   To check the continuity of $\tilde{H}$, we just need to check continuity at $(0,z)$ for each $z \in \Omega$. Fix $(0,z)\in [0,1]\times \Omega$.
    For any $t(\neq0)$ near $0$ and $h\in \mathbb{C}^n$ such that $z+h\in \Omega$, \[\tilde{H}(t,z+h)= \frac{F(E(1-t,z+h))}{t}=\frac{F(2tH(1,z))}{t},\ z\in \Omega\] Now, 
    \begin{equation}\label{Continuity of tilde H}
       \frac{F(2tH_1(z+h))}{t} =\frac{F(2tH_1(z+h))-F(2tH_1(z))}{t}+\frac{F(2tH_1(z))-F(0)}{t}
   \end{equation}
   Using the directional derivative of $F$ at $0$ in the direction of $H_1(z)$ we obtain that,
   \begin{equation*}
       \lim_{t\rightarrow0}\frac{F(2tH_1(z))-F(0)}{t}=DF(0)(2H_1(z))=2H_1(z)
   \end{equation*}
   Since $F$ is locally Lipschitz, we can find $M, r>0$ such that 
   \begin{equation*}
       ||F(z)-F(w)||<M||z-w|| , \ \ \forall z,w\in B_r(0).
   \end{equation*}
   Fix $z\in \Omega$. For any $(h,t)\in \Cpn\times \mathbb{R}^+$, such that $||h||,t$ are sufficiently small, we have the following, 
   \begin{equation}\label{bound on F by H1}
       ||\frac{F(2tH_1(z+h))-F(2tH_1(z))}{t}|| \leq \frac{1}{t}2tM||H_1(z+h)-H_1(z)||.              \end{equation}
   Clearly, $H_1$ is continuous on $\Omega$. From (\ref{bound on F by H1}) we get that,
       \begin{equation*}
          \lim_{(h,t)\rightarrow(0,0+)} \frac{F(2tH_1(z+h))-F(2tH_1(z))}{t}=0.
       \end{equation*}
       Also from (\ref{Continuity of tilde H}),
       \begin{equation*}
        \lim_{(h,t)\rightarrow(0,0+)} \frac{F(2tH_1(z+h))}{t}  =2H_1(z).
       \end{equation*}
    Therefore, $\tilde{H}$ is continuous at $(0,z)$. Since $z\in \Omega$ is arbitrary, $\tilde{H}$ is continuous on  $[0,1]\times \Omega$. Hence, for any $F\in \mathcal{E}(\Omega,\Cpn)$ there exists a continuous path $\gamma:[0,1]\rightarrow \mathcal{E}(\Omega,\Cpn)$ depending on $F$ given by \[\gamma(t)=\begin{cases}
        2H_{2t}  \  \ \text{if } 0\leq t\leq 1/2, \\
        \tilde{H}_{2t-1} \  \text{if } 1/2\leq t\leq 1.
    \end{cases}\]
    It satisfies  $\gamma(1)=F$, and $\gamma(0)=2Id$. Hence $\mathcal{E}(\Omega,\Cpn)$ is path connected.
    
  For the second part, choose $F\in \mathcal{E}_{\mathcal{R}}(\Omega,\Cpn)$. It is given that $H_t(\Omega) $ is Runge domain in $\Cpn$, for each $t\in[0,1]$.
   Here  $\tilde{H}_t=\frac{1}{t}F\circ E_{1-t}$ is a biholomorphism from $\Omega$ into $\Cpn$ for each $t\in(0,1]$.
    It follows from the construction of $E$ that 
    \[
    E_t(\Omega)=\begin{cases}
        H_{2t}(\Omega), \ \ &\text{if} \ 0\leq t \leq 1/2, \\
        2(1-t)H_1(\Omega) \ \ &\text{if} \ 1/2 \leq t \leq 1.
    \end{cases}.
    \]
    It is given that $H_t(\Omega)$ is Runge in $\Cpn$ for each $t\in [0,1]$. Also $2(1-t)H_1(\Omega)$ is Runge in $\Cpn$ when $t \neq 1$, as it is an image of Runge domain under scaling. Hence each $E_t(\Omega)$ is Runge in $\Cpn$ when $t \neq 1$.
    We claim that $\tilde{H}_t(\Omega)$ is Runge in $\Cpn$ for each $t\in [0,1]$.
   Assume $t\neq 1$. Since each  $\Omega$ and $E_t(\Omega)$ both are Runge domains in $\Cpn$ it follows that  $(E_t(\Omega),\Omega)$ is a Runge pair. Under a biholomorphism $F:\Omega\rightarrow F(\Omega)\subset\Cpn$, $(F(E_t(\Omega)), F(\Omega))$ remains a Runge pair in $\Cpn$. Since $F(\Omega)$ is Runge, $F(E_{t}(\Omega))$ is  Runge in $\Cpn$. This implies that $\tilde{H}_{1-t}(\Omega)=\frac{1}{1-t}F(E_t(\Omega))$ is Runge in $\Cpn$ whenever $t\neq 1$. If $t=1$, $\tilde{H}_0(\Omega)=2H_1(\Omega)$ is also a Runge domain by the given assumption. Therefore $\tilde{H}_t(\Omega)$ is Runge in $\Cpn$ for each $t\in[0,1]$.
    Again, by assumption each $H_t(\Omega)$ is Runge in $\Cpn$ and $H_0(z)=z$. Using Result \ref{Forstneric Rosay continuous isotopy}, $H_1$ can be approximated locally uniformly on $\Omega$ by holomorphic automorphisms of $\Cpn$. Therefore $\tilde{H}_0=2H_1$ can be approximated by automorphisms of $\Cpn$ locally uniformly on $\Omega$.
    Again, using Result \ref{Forstneric Rosay continuous isotopy}, we get that $\tilde{H}_1=F$ can be approximated by elements of $Aut(\Cpn)$ locally uniformly on $\Omega$. By assumption, for each $t\in [0,1]$, $F,H_t\in \mathcal{E}_\mathcal{R}(\Omega,\Cpn)$. Then  from the construction of $\gamma:[0,1]\rightarrow \mathcal{E}(\Omega,\Cpn)$ it follows that $\gamma(t)\in \mathcal{E}_{\mathcal{R}}(\Omega,\Cpn)$ for each $t\in [0,1]$. Hence, $\mathcal{E}_{\mathcal{R}}(\Omega,\Cpn)$ is path connected.
   
 \end{proof}

In this following lemma, we summarize some properties of a complex manifold that admits a holomorphic $\mathbb{R}_+$-complete vector field with a globally asymptotically stable equilibrium point. We will use this lemma in Corollary \ref{cor-Approximation by automorphism-on-FID}, Corollary \ref{Corollary need to prove discrete union embedding problem}, Proposition \ref{Runge embedding of phi-like domains}, Theorem \ref{Runge embedding of Stein manifolds which admits a R+complete holomorphic vector field $V$ which has a unique equilibrium point that is globally asymptotically stable}.
\begin{lemma}\label{properties of invariant open exhaustion}
    Let $M$ be a complex manifold. Let $M$ admits a $\mathbb{R}_+$-complete holomorphic vector field with a unique equilibrium point $p$, which is globally asymptotically stable on $M$. Then there exists a family of domains $\{\Omega_j\}_{j\in \mathbb{N}}$ in $M$ that satisfies the following properties.
    \begin{enumerate}
        \item $p\in \Omega_1\subset\hdots \subset\Omega_j\subset \Omega_{j+1}\subset \hdots \subset\bigcup_{j\in \mathbb{N}}\Omega_j=M$.
        \item $\bar{\Omega}_j\subset M$, for each $j\in \mathbb{N}$. In addition, for any neighborhood $V_p$ of $p$ there exists a $N_j\in \mathbb{N}$ such that $\theta_{N_j}(\Omega_j)\subset V_p$,  where $\theta:\mathbb{R}_+\times M\rightarrow M$ is the flow map of the vector field $V$.
        \item  For any $s\geq 0$, $\theta_s(\Omega_j)$ is invariant under $\theta_t$, for each $t\geq 0$.
        \item For each $s\geq 0$ and $j\in \mathbb{N}$, $\theta_s(\Omega_j)$ admits a $\mathbb{R}_+$-complete holomorphic vector field with a unique equilibrium point $p$, which is globally asymptotically stable on $\theta_s(\Omega_j)$. The vector field is given by $V|_{\theta_s(\Omega_j)}$.
        
        \item $(\Omega_m,\Omega_j)$ is a Runge pair for each $j,m$ such that $j\geq m.$
        
    \end{enumerate}
\end{lemma}
\begin{proof}
       It is given that $\theta:\mathbb{R}_+\times M\rightarrow M$ be the flow map corresponding to the $\mathbb{R}_+$-complete holomorphic vector field $V$.  Choose a compact exhaustion $\{K_j\}_{j\in \mathbb{N}}$ of $\Omega$ such that $K_j\subset K^{\circ}_{j+1}, \ \underset{j\in \mathbb{N}}{\bigcup}K_j=M$, and $p\in K_1^{\circ}$.  
    For each $j\in\mathbb{N}$, Let \[\Omega_j=\bigcup_{t\geq 0}\theta_t(K^{\circ}_j).\] Clearly, each $\Omega_j$ is open, and $\Omega_j\subset \Omega_{j+1}$, and $\bigcup_{j\in \mathbb{N}}\Omega_j=M$. Since $p\in K^{\circ}_1$ it follows that $p\in \Omega_1$. This proves $(1).$
    
   Since $p\in M$  is the globally asymptotically stable equilibrium point. Then $\theta$ satisfies that \begin{equation*}\label{Asymptotic limit tends to p}
        \lim_{t\tends \infty}\theta(t,z)=p, \  \forall z\in M.
    \end{equation*} Choose a connected neighborhood $V_p$ of $p$ in $M$ that satisfies $V_p\subset K_1^{\circ}$. By asymptotically stability condition at $p$, there exists a neighborhood $U_p$ of $p$ such that $\theta(t,z)\subset V_p$ for all $z\in U_p$ and $t\geq 0$.
   Fix $j\in \mathbb{N}$. For each $a\in K_j$ there exists a neighborhood $B_a$ of $a$, and a $t_a>0$ such that
    \[\theta(t_a,z)\in U_p ,\forall z\in B_a.\] Since $\theta(t_a,z)\in U_p$, this implies that 
    \[\theta(t,z)\in V_p, \  \forall z\in B_a.\]
Now $\{B_a\}_{a\in K_j}$ is an open cover of $K_j$. Since $K_j$ is compact it follows that there exists $\{a_1,\dots , a_l\}\subset K_j$, and $K_j\subset \cup_{i=1}^lB_{a_i}$. Choose $T_j=\{t_{a_1},\dots, t_{a_l}\}$. Then   \[\theta(t,z)\subset V_p, \ \forall t\geq T_j, \ \forall z\in K_j.\]  Hence $\theta([T_j,\infty)\times K_j)\subset \bar{V}_p\subset M$. Since $\theta$ is continuous, $\theta([0,T_j]\times K_j)$ is also compact. This implies that $\bigcup_{t\geq 0}\theta_t(K_j)\Subset M$. Therefore $\bar{\Omega}_j\subset M$. Since $\bar{\Omega}_j$ is compact in $M$, in the same way as of $K_j$, there exists a $N_j\in \mathbb{N}$ such that $\theta(N_j,z)\in V_p, \ \forall z\in \Omega_j$. Therefore $\theta_{N_j}(\bar{\Omega}_j)\subset V_p$. This proves $(2)$.\\
 Fix $j\in \mathbb{N}$. We now check that each $\Omega_j$ is invariant under $\theta_t(t\geq0)$. For any $w\in \Omega_j$, $w=\theta_{t_0}(u)$ for some $t_0\geq 0$ and $u\in K^{\circ}_j$. Then $\theta_t(w)=\theta_t(\theta_{t_0}(u))=\theta_{t+t_0}(u)\in \theta_{t+t_0}(K^{\circ}_j)\subset \Omega_j$. So $\Omega_j$ is invariant under $\theta_t$, for each $t\geq 0$. Fix $s\geq 0$. Let $w\in \theta_s(\Omega_j)$. Then $w=\theta_s(x)$ for some $x\in \Omega_j$. For any $t\geq 0,$\[\theta_t(w)=\theta_t(\theta_s(x))=\theta_s(\theta_t(x))\in \theta_s(\Omega_j).\] Therefore $\theta_s(\Omega_j)$ is invariant under  $\theta_t$, for each $t\geq 0$. This proves $(3)$.

  Since for each $s\geq 0, j\in \mathbb{N}$, $\theta_s(\Omega_j)$ is invariant under $\theta_t$, for any $ t\geq 0$, and $p\in \theta_s(\Omega_j)$. So that $\theta_s(\Omega_j)$ and $V|_{\theta_s(\Omega_j)}$ satisfy all three conditions in Definition \ref{Definition- Manifolds that admits a vector field globally asymptotically stable}.
    This proves $(4)$.

      We choose a coordinate chart $(\psi,U)$ in a neighborhood of $p$ and fix $j,m\in \mathbb{N}$ such that $j>m$. From property $(2)$, we get that there exists a $N\in \mathbb{N}$ such that $\theta_{N}(\Omega_j)\subset U$. Using Lemma \ref{Rungepair biholomorphic invariant} it follows that $(\psi(\theta_N(\Omega_m)),\psi(\theta_N(\Omega_j))$ is a Runge pair if and only if $(\Omega_m,\Omega_j)$ is a Runge pair. Therefore, it is enough to show  $(\psi(\theta_N(\Omega_m)),\psi(\theta_N(\Omega_j))$ is a Runge pair. We know that $\psi(\theta_N(\Omega_j))\subset \Cpn$ and it admits the vector field $\psi_*(V|_{\theta_N(\Omega_j)})$, where $\psi_*V|_{\theta_N(\Omega_j)}$ is defined by $(\psi_*V|_{\theta_N(\Omega_j)})_{\psi(z)}=D\psi_z(V_z)$, for all $z\in \theta_N(\Omega_j)$. It is well known that if $\dot{\theta}$ is the flow of $\psi_*V|_{\theta_N(\Omega_j)}$ on $\theta_N(\Omega_j)$, then it satisfies that\[ \dot{\theta}(t,\psi(z))=\psi(\theta(t,z)),\ \forall t\geq 0, \forall z\in \theta_N(\Omega_j).\] Since $\theta_N(\Omega_m)$ is invariant under $\theta_t,$ for each $\forall t\geq 0$, it follows that $\psi(\theta_N(\Omega_m))$ is invariant under $\dot{\theta}_t$, for each $t\geq 0$. Now we can use Result \ref{Rungepairphilikedomains} to conclude that $(\psi(\theta_N(\Omega_m)),\psi(\theta_N(\Omega_j))$ is a Runge pair. Therefore, $(\Omega_m,\Omega_j)$ is a Runge pair and this proves $(5)$.

\end{proof}

\begin{corollary}\label{cor-Approximation by automorphism-on-FID}
    Let $\Omega$ be a Runge domain on $\Cpn \ (n\geq 2)$. Assume that $\Omega$ admits a $\mathbb{R}_+$-complete holomorphic vector field with a unique equilibrium point $p$, which is globally asymptotically stable on $\Omega$. Then $\mathcal{E}_{\mathcal{R}}(\Omega,\Cpn)\subset\overline{Aut(\Cpn)}|_{\Omega}$ in $\mathcal{E}(\Omega,\Cpn)$.
\end{corollary}
\begin{proof} Let $V$ be the holomorphic vector field as given in the assumption, and $\theta:\mathbb{R}_+\times \Omega\rightarrow \Omega$ be the corresponding flow map. Without loss of generality, we may assume that $p=0\in \Omega$ is the asymptotically stable equilibrium point of $V$.
   Using Lemma \ref{properties of invariant open exhaustion}, we get a sequence $\{\Omega_j\}_{j\in \mathbb{N}}$ of domains in $\Omega$ with properties $(1),(2),(3),(4),(5)$ in Lemma \ref{properties of invariant open exhaustion}.
   Again by $(3)$ in Lemma \ref{properties of invariant open exhaustion}, each $ \Omega_j$ is invariant under $\theta_t$, for all $ t\geq 0$. Therefore using Result  \ref{Rungepairphilikedomains}, we get that $(\Omega_j,\Omega)$ is a Runge pair. Since $\Omega$ is a Runge domain in $\Cpn$, it follows that $\Omega_j$ is also Runge in $\Cpn$. We choose $r>0$ such that the closed ball  $\overline{B_r(0)}\subset \Omega$. Fix $j\in \mathbb{N}$. By $(2)$ in Lemma \ref{properties of invariant open exhaustion}, we know that $\bar{\Omega}_j\subset \Omega$, and there exists a $N_j\in \mathbb{N}$ such that $\theta_{N_j}(\Omega_j)\subset B_r(0).$ Since $\Omega_j$ is invariant under $\theta_t$, $\forall t\geq 0$, we construct an isotopy $H_j:[0,1]\times \Omega_j\rightarrow \Omega_j$ by 
    \[H_j(t,z)=\theta(tN_j,z), \ z\in\Omega_j.\]    Then $H_j$ is a continuous isotopy of biholomorphic maps from $\Omega_j$ into $\Omega_j$, $H_j(0,z)=z$. Since $\Omega_j$ admits the vector field $V|_{\Omega_j}$, using Result \ref{Rungepairinsemigroups}, we get that $(\theta_{tN_j}(\Omega_j),\Omega_j)=(H_{j,t}(\Omega_j),\Omega_j)$ is a Runge pair. Since $\Omega_j$ is a Runge domain in $\Cpn$, it follows that each $H_{j,t}(\Omega_j)$ is a Runge domain in $\Cpn$. Now there exists $j_0\in \mathbb{N}$ such that $\overline{B_r(0)}\subset \Omega_j, \forall j\geq j_0$, and $H_j(1,\cdot)$ satisfies that \[H_j(1,z)\in B_r(0)\subset \Omega_{j}, \ \forall z\in \Omega_{j},\ j\geq j_0.\] 
    Note that, $B_r(0)$ is a star-shaped domain with center at $0$. Now using Theorem \ref{Starshaped}, we obtain that $\mathcal{E}_{\mathcal{R}}(\Omega_j,\Cpn)\subset\overline{Aut(\Cpn)}|_{\Omega_j}$ in $\mathcal{E}(\Omega_j,\Cpn)$  for each $j\geq j_0 \in \mathbb{N}$. Let $F\in \mathcal{E}_{\mathcal{R}}(\Omega,\Cpn)$. For any compact $K\subset\Omega$, there exists a integer $k\geq j_0$ such that $K\subset \Omega_k$. By assumption, $F(\Omega)$ is Runge in $\Cpn$. Also, we know that $(F(\Omega_k),F(\Omega))$ is a Runge pair. Hence $F(\Omega_k)$ is Runge in $\Cpn$, and $F|_{\Omega_k}\in \mathcal{E}_{\mathcal{R}}(\Omega_k,\Cpn)$. Therefore $F$ can be approximated by automorphisms of $\Cpn$ uniformly on $K$. Since $K\subset\Omega$ is arbitrary, it follows that $\overline{Aut(\Cpn)}|_{\Omega}$ in $\mathcal{E}(\Omega,\Cpn)$ contains $\mathcal{E}_{\mathcal{R}}(\Omega,\Cpn)$.
  \end{proof}

The following corollary is an analogous result of Proposition \ref{Continuous isotopy and approximation of Y valued biholomorphisms} to approximate biholomorphic maps with Runge image in $\Cpn$. The proof depends on Corollary \ref{cor-Approximation by automorphism-on-FID}.
\begin{corollary}\label{Corollary need to prove discrete union embedding problem}
  Let $\Omega\subset\tilde{\Omega}$ be two domains in a complex manifold $M$ $(dim M=n\geq2)$ such that  $\Omega$ admits a $\mathbb{R}_+$-complete holomorphic vector field $V$ with a unique equilibrium point which is globally asymptotically stable on $\Omega$. Assume that $(\Omega,\tilde{\Omega})$ is a Runge pair and $\tilde{\Omega}$ admits a Runge embedding in $\Cpn$. Then $\mathcal{E}_{\mathcal{R}}(\Omega,\Cpn)\subset\overline{\mathcal{E}_{\mathcal{R}}(\tilde{\Omega},\Cpn)}|_{\Omega}$ in $\mathcal{E}(\Omega,\Cpn)$.
\end{corollary}
\begin{proof}
    Let $\phi:\tilde{\Omega}\rightarrow \phi(\tilde{\Omega)}\subset\Cpn$ be the Runge embedding of $\tilde{\Omega}$.
    Choose any $F\in \mathcal{E}_{\mathcal{R}}(\Omega,\Cpn)$.
    We have a biholomorphic map $\phi\circ F^{-1}:F(\Omega)\rightarrow\phi(\Omega)$.
    Since $(\Omega,\tilde{\Omega})$ is a Runge, the pair of domains $(\phi(\Omega),\phi(\tilde{\Omega}))$ is also a Runge pair. Since $\phi(\tilde{\Omega})$ is a Runge domain in $\Cpn$, we get that $\phi(\Omega)$ is Runge in $\Cpn$. Hence,
     $$(\phi\circ F^{-1})^{-1}=F\circ\phi^{-1}:\phi(\Omega)\rightarrow F(\Omega)$$ is a biholomorphic map between two Runge domains in $\Cpn$.
    Clearly, $\phi(\Omega)$ admits the vector field $\phi_*V$, which is $\mathbb{R}_+$-complete and it has a unique globally asymptotically stable equilibrium point, and $F\circ \phi^{-1}$ is a biholomorphism from $\phi(\Omega)$ onto $F(\Omega)$. Using Corollary \ref{cor-Approximation by automorphism-on-FID}, there exists a sequence $\{\Psi_k\}_{k\in\mathbb{N}}$ of automorphisms of $\mathbb{C}^n$ such that $\Psi_k\rightarrow F\circ\phi^{-1}$ locally uniformly on $\phi(\Omega)$ as $k\rightarrow \infty $. This implies that $\Psi_k\circ \phi\rightarrow F$ locally uniformly on $\Omega$ as $k\rightarrow\infty$.
    Consider the maps 
    \[F_k=\Psi_k\circ \phi :\tilde{\Omega}\rightarrow\Cpn.
    \] 
    Then, $F_k$ converges to $F$ locally uniformly on $\Omega$.  We also obtain that $F_k(\tilde{\Omega})=\Psi_k(\phi(\tilde{\Omega}))$ is Runge in $\Cpn$ as $\phi(\tilde{\Omega})$ is Runge in $\Cpn$, and $\Psi_k\in Aut(\Cpn)$. Therefore, $\mathcal{E}_{\mathcal{R}}(\Omega,\Cpn)\subset\overline{\mathcal{E}_{\mathcal{R}}(\tilde{\Omega},\Cpn)}|_{\Omega}$ in $\mathcal{E}(\Omega,\Cpn)$.
\end{proof}
\begin{remark}
In particular, if $\tilde{\Omega}$ is a Runge domain in $\Cpn$ then we can remove the assumption that $\tilde{\Omega}$ admits a Runge embedding in $\Cpn$. So that it provides a suffiecient condition on approximation of injective holomorphic maps by injective holomorphic maps of larger domain.
\end{remark}
\begin{proposition}\label{Runge embedding of phi-like domains}
    Let $\Omega $ be a Stein domain in $\Cpn \ (n\geq 2)$ and $\Omega$ admits a $\mathbb{R}_+$-complete vector field $V$ which has a unique equilibrium point $p$, which is globally asymptotically stable on $\Omega$. Then $\Omega$ admits a holomorphic Stein and Runge embedding in $\Cpn$.
\end{proposition}
\begin{proof}
   Without loss of generality, we may assume that $p=0$. By Lemma \ref{properties of invariant open exhaustion}, there exists a sequence of domains $\{\Omega_j\}_{j\in \mathbb{N}}$  that satisfies properties $(1),(2),(3),(4),(5)$ given in Lemma \ref{properties of invariant open exhaustion}.
    Fix $j,m\in \mathbb{N}$ such that $j>m$. We now show that $\mathcal{E}_{\mathcal{R}}(\Omega_m,\Cpn)\subset \overline{\mathcal{E}_{\mathcal{R}}(\Omega_j,\Cpn)}|_{\Omega_m}$ in $\mathcal{E}(\Omega_m,\Cpn) $. 
      From $(4)$ in Lemma \ref{properties of invariant open exhaustion}, we know that $(\Omega_m,\Omega_j)$ is a Runge pair. Again, $\Omega_m$ admits the vector field $V|_{\Omega_m}$, which is $\mathbb{R}_+$-complete on $\Omega_m$ and has a unique equilibrium point which is  globally asymptotically stable.
    To apply Corollary \ref{Corollary need to prove discrete union embedding problem}, we need to show that each $\Omega_j$ admits a Runge embedding in $\Cpn$. Since $\bar{\Omega}_j\subset \Omega$ there exists $N\in \mathbb{N}$ such that $\theta_{N}(\Omega_j)\subset B_r(0)\subset\Omega$ (see $(2)$, Lemma \ref{properties of invariant open exhaustion}). By Result \ref{Rungepairphilikedomains}, $(\theta_{N}(\Omega_j),\Omega)$ is a Runge  pair. Therefore, any holomorphic function $h$ on $\theta_{N}(\Omega_j)$ can be approximated by a holomorphic function $\{h_k\}_{k\in \mathbb{N}}$ in $\mathcal{O}(\Omega)$ locally uniformly in $\theta_N(\Omega_j)$. Hence, $h_k|_{B_r(0)}$ converges to $h$ locally uniformly in $\theta_N(\Omega_j)$. Since $B_r(0)$ is a Runge domain in $\Cpn$, each $h_k|_{B_r(0)}$ is a limit of entire maps locally uniformly on $\theta_N(\Omega_j)$. It implies that $\theta_N(\Omega_j)$ is a Runge domain in $\Cpn$. So that $\theta_{N}: \Omega_j\rightarrow\Cpn$ is the corresponding Runge embedding of $\Omega_j$ into $\Cpn$. Now using Corollary \ref{Corollary need to prove discrete union embedding problem} we obtain that $\mathcal{E}_{\mathcal{R}}(\Omega_m,\Cpn)\subset \overline{\mathcal{E}_{\mathcal{R}}(\Omega_j,\Cpn)}|_{\Omega_m}$ in $\mathcal{E}(\Omega_m,\Cpn) $. Since this is true for any $j,m$, using Theorem \ref{Union embedding in domain with Density}, we conclude that $\Omega$  admits a Stein and Runge embedding in $\Cpn$.

\end{proof}


We now present a proof of Theorem~\ref{Runge embedding of Stein manifolds which admits a R+complete holomorphic vector field $V$ which has a unique equilibrium point that is globally asymptotically stable}.
\begin{proof}[Proof of Theorem \ref{Runge embedding of Stein manifolds which admits a R+complete holomorphic vector field $V$ which has a unique equilibrium point that is globally asymptotically stable}]

By Lemma \ref{properties of invariant open exhaustion}, we get a sequence of domains $\{\Omega_j\}_{j\in \mathbb{N}}$ in $X$ that satisfy the conclusions $(1),(2),(3),(4),(5)$ of Lemma \ref{properties of invariant open exhaustion}. Let us choose a chart $(\psi,V)$ such that $\psi(p)=0\in \Cpn$. Fix $j\in \mathbb{N}$. Then, there exists a $k\in \mathbb{N}$ such that $\bar{\Omega}_j\subset \Omega_k$. 
    By $(2)$ in Lemma \ref{properties of invariant open exhaustion}, there exists  $N_k\in \mathbb{N}$ such that $\theta_{N_k}(\Omega_k)\subset V$. Therefore, we get that $\overline{\psi(\theta_{N_k}(\Omega_j))}\subset \psi(\theta_{N_k}(\Omega_k))\subset \Cpn$.  Since $(\Omega_j,\Omega_k)$ is a Runge pair, the pair of domains $(\psi(\theta_{N_k}(\Omega_j)),\psi(\theta_{N_k}(\Omega_k))$ is also a Runge pair. Now, in the same way as in our proof of Proposition \ref{Runge embedding of phi-like domains}, we can show that $\psi(\theta_{N_k}(\Omega_j))$ has a Runge embedding in $\Cpn$.  Let $\Phi:\psi(\theta_{N_k}(\Omega_j))\rightarrow \Cpn$ be a Runge embedding of $\psi(\theta_{N_k}(\Omega_j))$. Then, \[
    \Phi\circ\psi\circ\theta_{N_k}:\Omega_j\rightarrow \Cpn
    \]
    is a Runge embedding of $\Omega_j$. Let take any $m\in \mathbb{N}$ such that $\Omega_m\subset \Omega_j$. Also, $(\Omega_m,\Omega_j)$ is a Runge pair and $\Omega_j$ admits the $\mathbb{R}_+$-complete vector field $V|_{\Omega_j}$. So that, using Corollary \ref{Corollary need to prove discrete union embedding problem}, we get that $\mathcal{E}_{\mathcal{R}}(\Omega_m,\Cpn)\subset \overline{\mathcal{E}_{\mathcal{R}}(\Omega_j,\Cpn)}|_{\Omega_m}$ in $\mathcal{E}(\Omega_m,\Cpn)$. This is true for any $m,j\in \mathbb{N}$ such that $m<j$. Using Theorem \ref{Union embedding in domain with Density}, we conclude that $X$ admits a Stein and Runge embedding in $\Cpn$.

\end{proof}

Next, we present a proof of Theorem \ref{Runge embedding of increasing sequence of Stein domains}. 
 It uses Corollary \ref{Corollary need to prove discrete union embedding problem}, Theorem \ref{Union embedding in domain with Density}, and Theorem \ref{Runge embedding of Stein manifolds which admits a R+complete holomorphic vector field $V$ which has a unique equilibrium point that is globally asymptotically stable}.
    \begin{proof}[Proof of Theorem \ref{Runge embedding of increasing sequence of Stein domains}]
    Fix $j\in \mathbb{N}$.
        By assumption $(ii)$ and Theorem \ref{Runge embedding of Stein manifolds which admits a R+complete holomorphic vector field $V$ which has a unique equilibrium point that is globally asymptotically stable} we get that  $\Omega_{j+1}$ admits a Stein and Runge embedding in $\Cpn$. Again $\Omega_{j}$   admits a $\mathbb{R}_+$-complete holomorphic vector field with a unique equilibrium point which is globally asymptotically stable. By assumption, $(\Omega_j,\Omega_{j+1})$ is a Runge pair. We now use corollary \ref{Corollary need to prove discrete union embedding problem} to conclude that $\mathcal{E}_{\mathcal{R}}(\Omega_j,\Cpn)\subset\overline{\mathcal{E}_{\mathcal{R}}(\Omega_{j+1},\Cpn)}|_{\Omega_j}$ in $\mathcal{E}(\Omega_j,\Cpn)$. Assumption $(ii)$ and Theorem \ref{Runge embedding of Stein manifolds which admits a R+complete holomorphic vector field $V$ which has a unique equilibrium point that is globally asymptotically stable} together implies that $\mathcal{E}_{\mathcal{R}}(\Omega_j,\Cpn)\neq\emptyset$ for each $j\in \mathbb{N}$. Again, by assumption, each $(\Omega_j,\Omega_{j+1})$ is a Runge pair, and each $\Omega_j$ is Stein. Therefore, using Result \ref{Stein property of increasing union}, $X$ is itself a Stein manifold. Now, using Theorem \ref{Union embedding in domain with Density}, we conclude that $X$ admits a Stein and Runge embedding in $\Cpn$. Additionally, we get that $\mathcal{E}_{\mathcal{R}}(\Omega_j,\Cpn)\subset\overline{\mathcal{E}_{\mathcal{R}}(X,\Cpn)}|_{\Omega_j}$ in $\mathcal{E}(\Omega_j,\Cpn)$ for any $j\in \mathbb{N}$.
        
    \end{proof}

\section{$(\mathbb{R_+},+)$ actions on Stein manifolds and embeddings}\label{sec-transferSemigroup}
 In this section, we state and prove a theorem about inducing an $(\rea,+)$-action from an $(\rea_+, +)$-action which might be of independent interest. Here $\mathbb{R}_+=\{t\geq0:t\in \mathbb{R}\}.$
    \begin{theorem}\label{Embedding of Steins having R_+ action into Steins having R actions}
    Let $X$ be a Stein manifold and $\theta:\mathbb{R}_+\times X \rightarrow X$ be a continuous semigroup action of $(\mathbb{R}_+,+)$ on $X$ and  $\theta(t,\cdot)=\theta_t:X\rightarrow X$ is holomorphic, for each $t\geq 0$. Then there exists a Stein manifold $\tilde{X}$, a continuous $(\mathbb{R},+)$ action $\tilde{\theta}:\mathbb{R}\times \tilde{X}\rightarrow \tilde{X}$ and an injective holomorphic map $f:X\rightarrow \tilde{X}$ such that $\tilde{\theta}_t\circ f=f\circ\theta_t$ on $X$ for all $t\geq 0$.
\end{theorem}
The above theorem helps us to get an analogous result of Theorem \ref{Embedding of Stein manifold using R-actions} using action of the semigroup $(\mathbb{R}_+,+)$ without technical difficulties.
In the following proof, we construct a quotient space by giving an equivalence relation on $X\times \mathbb{R}_+$, where $\mathbb{R}_+$ is endowed with the discrete topology.
 Then we show that the quotient space has a complex manifold structure, and we define a $(\mathbb{R},+)$-action on the quotient space.
\begin{proof}
Let us consider the topological space $X\times \mathbb{R}_+$ with $\mathbb{R}_+$ given with the discrete topology. We define a relation $'\sim '$ on $X\times \mathbb{R}_+$ by \[(x,t)\sim (y,s) \ \ \text{if} \ \ \theta_t(y)=\theta_s(x).\]
It is an equivalence relation on $X\times \mathbb{R}_+$. Let $[x,t]$ denote the equivalence class of $(x,t)\in X\times\mathbb{R}_+$. Take the quotient space $\tilde{X}=X\times \mathbb{R}_+/\sim$, and let $q:X\times \mathbb{R}_+\rightarrow \tilde{X}$ be the quotient map given by $q(x,t)=[x,t],\ \forall x\in X, \ \forall t\in \mathbb{R}_+$. Using result \ref{injectivity of continuous semigroup action}, we get that $\theta_t$ is one-one, for each $t\geq 0$. Fix $s\in \mathbb{R}_+$. For any open set $U\subset X$, the set \[q^{-1}\{[x,s]:x\in U\}=\bigcup_{0\leq t<s}(\theta^{-1}_{s-t}\{U\}\times\{t\}) \cup (\bigcup_{t\geq s}(\theta_{t-s}(U)\times \{t\}))\] is open in $X\times \mathbb{R}_+$. Hence $\{[x,s]:x\in U\}$ is open in $\tilde{X}$. Let $V_s=\{[x,s]:x\in X\}\subset \tilde{X} \ (s\geq 0)$. Then each $V_s$ is a open set in $\tilde{X}$. For $0\leq s_1<s_2$,\[  [x,s_1]\in V_{s_1}\Rightarrow [x,s_1]=[\theta_{s_2-s_1}(x),s_2]\in V_{s_2}.\] This implies that $V_{s_1}\subset V_{s_2} $. Let us define a family of maps $f_s:X\rightarrow \tilde{X} \ (s\geq 0)$ by $f_s(x)=[x,s], \ \forall x\in X$. Fix $s\geq 0$. Now \[f_s(x_1)=f_s(x_2)\Rightarrow[x_1,s]=[x_2,s]\Rightarrow \theta_s(x_1)=\theta_s(x_2)\Rightarrow x_1=x_2.\] This implies that each $f_s$ is one-one map from $X$ onto the open set $V_s$. Also, for any open set $U\subset X$, $f_s(U)=\{[x,s]:x\in U\}$ is open. Therefore $f_s$ is a homeomorphism onto the subspace $V_s$ of $\tilde{X}$. For $0\leq s<t$, we have $V_s\cap V_t=V_s$, $f^{-1}_t(V_s)=\theta_{t-s}(X)$. The transition maps 
\[f^{-1}_t\circ f_s:f^{-1}_s(V_s)\rightarrow f^{-1}_t(V_s), \ f^{-1}_t\circ f_s: X\rightarrow \theta_{t-s}(X),\]
\[ f^{-1}_t\circ f_s(x)=f^{-1}_t([x,s])= f^{-1}_t([\theta_{t-s}(x),t])=\theta_{t-s}(x),\] and
\[f^{-1}_s\circ f_t:f^{-1}_t(V_s)\rightarrow f^{-1}_s(V_s), \ f^{-1}_s\circ f_t:\theta_{t-s}(X)\rightarrow X,\]
\[\Rightarrow f^{-1}_s\circ f_t(y)=f^{-1}_s([y,t])=f^{-1}_s([\theta^{-1}_{t-s}(y),s])=\theta^{-1}_{t-s}(y),\] are holomorphic. So that $(f^{-1}_s,f_s(V_s))$ are the manifold valued charts, for each $s\geq 0$.
  Since $f_s(X)\subset f_t(X)$ for any $0\leq s<t$, we obtain that $\tilde{X}=\bigcup _{k\in \mathbb{N}}f_k(X)$. We know that $X$ is second countable. It implies that $f_k(X)$ is second countable for each $k\in \mathbb{N}$. Since $\tilde{X}$ is countable union of $f_k(X)(k\in \mathbb{N})$, it follows that $\tilde{X}$ is second countable. It is easy to check that $\tilde{X}$ is Hausdorff as each $f_k(X)$ is Hausdorff. Therefore $\tilde{X}$ has a complex manifold structure and each $f_s:X\rightarrow \tilde{X}$ is a holomorphic map.  Define an $\mathbb{R}$-action $\tilde{\theta}:\mathbb{R}\times \tilde{X}\rightarrow \tilde{X}$ given by \[\tilde{\theta}(t,[x,s])=\begin{cases}
         [\theta_t(x),s]  \ \ \text{when} \ t\geq 0,\\
         [x,s-t]    \ \ \text{when} \ t\leq 0.
     \end{cases}\]
     It is easy to 
     check that $\tilde{\theta}$ is a well-defined map.
     If $t+t'\geq 0, t,t'\geq 0$ then
     \[ \tilde{\theta}(t+t',[x,s])=[\theta_{t+t'}(x),s]=[\theta_t(\theta_{t'}(x)),s]=\tilde{\theta}(t,[\theta_{t'}(x),s])=\tilde{\theta}(t,\tilde{\theta}(t',[x,s])).\]
     If $t+t'\geq 0, t\geq 0, t'\leq 0$ then 
     \[ \tilde{\theta}(t,\tilde{\theta}(t',[x,s])=\tilde{\theta}(t,[x,s-t'])=[\theta_t(x),s-t']=[\theta_{t+t'}(x),s]=\tilde{\theta}(t+t',[x,s]), \]
     \[ \tilde{\theta}(r',\tilde{\theta}(r,[x,s])=\tilde{\theta}(r',[\theta_r(x),s])=[\theta_r(x),s-r']=[\theta_{r+r'}(x),s]=\tilde{\theta}(r+r',[x,s])\]
     The other cases for $r,r'$ follows in the same way and it follows that $\tilde{\theta}$ is an action of $(\mathbb{R},+)$ on $\tilde{X}$. We now show that for each $t\in \mathbb{R}$, $\tilde{\theta}_t:\tilde{X}\rightarrow\tilde{X}$ is holomorphic. Fix $t\geq 0$ and choose an open set $V_s\subset\tilde{X}$. Then \[(f_s^{-1}\circ\tilde{\theta}_t\circ f_s):X\rightarrow X, \ (f_s^{-1}\circ\tilde{\theta}_t\circ f_s)(x)=\theta_t(x),\] is a holomorphic map on $X$. Hence $\tilde{\theta}_t$ is holomorphic on $V_s$. Since this is true for any $V_s$ it follows that $\tilde{\theta}_t$ is holomorphic on $\tilde{X}$. For $t\leq 0$, similarly we can get that $\tilde{\theta}_t$ is holomorphic on $\tilde{X}$. We now check that $\tilde{\theta}:\mathbb{R}\times \tilde{X}\rightarrow\tilde{X}$ is continuous.
     The following map \[\sigma:\mathbb{R}_+\times (X\times\mathbb{R}_+)\rightarrow \tilde{X}, \ \sigma(t,(x,s))=[\theta_t(x),s] , \ \forall t\in \mathbb{R}_+,\forall (x,s)\in X\times \mathbb{R}_+,\] is continuous. In addition, $\sigma$ satisfies that $\sigma=\tilde{\theta}\circ(id_{\mathbb{R}_+}\times q)$ on $\mathbb{R}_+\times (X\times\mathbb{R}_+)$. Since $\mathbb{R}_+$ with the usual topology is locally compact, it follows that $id_{\mathbb{R}_+}\times q: \mathbb{R}_+\times (X\times\mathbb{R}_+)\rightarrow \mathbb{R}_+\times \tilde{X} $ is a quotient map. Now using the universal property of quotient maps, we get that $\tilde{\theta}$ is continuous on $\mathbb{R}_+\times \tilde{X}$. Let $(r,[x,s])\in \mathbb{R}\times \tilde{X}.$ If $r> 0$ then $\tilde{\theta}$ is continuous at $(r,[x,s])$. Assume $r\leq 0$. Let choose $N\in \mathbb{N}$ such that $N+r>0$. Let us write \[\tilde{\theta}(r,[x,s])=\tilde{\theta}_{-N}\circ \tilde{\theta}(N+r,[x,s]).\] Since $\tilde{\theta}_{-N}:\tilde{X}\rightarrow \tilde{X}$ is holomorphic and $\tilde{\theta}$ is continuous at $(N+r,[x,s])\in \mathbb{R}_+\times \tilde{X}$, it follows that $\tilde{\theta}$ is continuous at $(r,[x,s])$. Therefore we conclude that $\tilde{\theta}$ is continuous on $\mathbb{R}\times \tilde{X}$.
     
     Let us choose $f:X\rightarrow \tilde{X}$ to be $f=f_0$ on $X$. Then for any $t\geq 0$, we obtain that  \[\tilde{\theta}_t([x,0])=[\theta_t(x),0]\Rightarrow \tilde{\theta}_t\circ f_0(x)=f_0\circ \theta_t(x) \Rightarrow \tilde{\theta}_t\circ f(x)=f\circ\theta_t(x) , \ \forall x\in X\]
     We now check that $\tilde{X}$ is Stein.
     Consider $V_0=\{[x,0]:x\in X\}\subset \tilde{X}$. Since $X$ and $V_0$ is biholomorphic to each other $(f:X\rightarrow V_0)$, it follows that $V_0$ is a Stein domain $\tilde{X}$. It is easy to check that $V_0$ is invariant under $\tilde{\theta}_t$, for each $t\geq 0$, and  $\tilde{\theta}_{s_2}(V_0)\subset \tilde{\theta}_{s_1}(V_0)$ for $0\leq s_1\leq s_2$. Now using Theorem \ref{Checking Runge pairs for R-actions} it follows that $(\tilde{\theta}_{s_2}(V_0),\tilde{\theta}_{s_1}(V_0))$ is a Runge pair for any $0\leq s_1\leq s_2$ . Hence $(\tilde{\theta}_{1}(V_0),V_0)$ is a Runge pair. For each $j\in \mathbb{N}$, $\tilde{\theta}_{-j}:\tilde{X}\rightarrow \tilde{X}$ is a holomorhic automorphism of $\tilde{X}$, and this implies that $(\tilde{\theta}_{-j}(V_0),\tilde{\theta}_{-(j+1)}(V_0))=(V_j,V_{j+1})$ is a Runge pair. Now we know that $\tilde{X}=\bigcup_{j\in \mathbb{N}}V_j$, $V_j\subset V_{j+1}$ and $(V_j,V_{j+1})$ is Runge pair, and $V_j$ Stein for each $j\in \mathbb{N}$. Using Result \ref{Stein property of increasing union}, we get that $\tilde{X}$ is a Stein manifold, and $(V_0,\tilde{X})$ is a Runge pair. Therefore $(f(X),\tilde{X})$ is a Runge pair.
     \end{proof}
     \begin{remark}
         If the $\mathbb{R}_+$-action $\theta$ on $X$ extends as an action of $(\mathbb{R},+)$ on $X$, then it follows that the injective holomorphic map $f:X\rightarrow\tilde{X}$ is onto. There exists a large class of examples of such manifolds in \cite{AhernFloresRosay} due to Ahern, Flores and Rosay. 
     \end{remark}
     
     \begin{proof} [Proof of Theorem~\ref{thm-semigroupAction}]
      If $\mathcal{E}_{\mathcal{R}}(\Omega,W)=\emptyset$ then there is nothing to prove. Let us assume that $\Omega\neq \emptyset$  and $\mathcal{E}_{\mathcal{R}}(\Omega,W)\neq \emptyset$. 
      By Result \ref{injectivity of continuous semigroup action}, each $\theta_t:X\rightarrow X$ is an injective holomorphic map.
         Using Theorem \ref{Embedding of Steins having R_+ action into Steins having R actions}, we get a Stein manifold $\tilde{X}$  and an injective holomorphic map $f:X\rightarrow \tilde{X}$ such that $(f(X),\tilde{X})$ is a Runge pair. Also we have $\tilde{\theta}_t\circ f=f\circ\theta_t \ (t\geq 0)$ on $X$. Since $\Omega$ is invariant under $\theta_t \ (t\geq 0)$ so that $f(\Omega)$ is invariant under $\tilde{\theta}_t (t\geq 0)$.
          Since each orbit  $\{\theta_t(x):t\in \mathbb{R}_+\}(x\in X)$ in $X$ intersects $\Omega$, it follows that each orbit $\{\tilde{\theta}_t(z):t\in \mathbb{R}\} \ (z\in \tilde{X})$ intersects $f(\Omega)$.
         We now use Theorem \ref{Embedding of Stein manifold using R-actions} to conclude that $\bigcup_{t\geq 0}\tilde{\theta}_{-t}(f(\Omega))=\tilde{X}$ admits a holomorphic Stein and Runge embedding in $W$. We also get that $\mathcal{E}_{\mathcal{R}}(f(\Omega),W)\subset \overline{\mathcal{E}_{\mathcal{R}}(\tilde{X},W)}|_{f(\Omega)}$ in $\mathcal{E}(f(\Omega),W)$.
         Let $\Psi: \tilde{X}\rightarrow W$ be the Stein and Runge embedding of $\tilde{X}$ in $W$. Since $(f(X),\tilde{X})$ is a Runge pair, by Lemma \ref{Rungepair biholomorphic invariant}, we get that $(\Psi(f(X)),\Psi(\tilde{X}))$ is also a Runge pair. We also get that  $\Psi(f(X))$ is Runge in $W$ as $\Psi(\tilde{X})$ is Runge in $W$. Therefore $\Psi\circ f:X\rightarrow W$ is the Runge embedding of $X$ in $W$. Also, using the map $f:X\rightarrow\tilde{X}$, we can get $\mathcal{E}_{\mathcal{R}}(\Omega,W)\subset\overline{\mathcal{E}_{\mathcal{R}}(X,W)}$ in $\mathcal{E}(\Omega,W)$.
     \end{proof}

The following example can be thought as an analogue of abstract basin of attraction that was given by Forn\ae ss and Stens{\o}nes in \cite{FornaessStensones}.
     
     \begin{example}
     Let $U$ be a Stein and Runge domain in $\Cpn \ (n\geq2)$ and $V$ is a $\mathbb{R}$-complete holomorphic vector field on $\Cpn$. Assume that $U$ is invariant under $\theta_t$ for $t\geq 0$ where $\theta:\mathbb{R}\times\Cpn\rightarrow\Cpn$ is the flow of the holomorphic vector field $V$.
     We define the relation $'\sim'$ on $U\times \mathbb{R}_+$, $(x,s)\sim (y,t)$ if $\theta_t(x)=\theta_s(y)$.
     $X:=U\times \mathbb{R}_+/\sim \ =\{[x,s]:x\in U,s\in \mathbb{R}_+\}$. Here, $\mathbb{R}_+$ is endowed with the discrete topology. Consider the family of maps $f_t:U\rightarrow X$ given by $f_t(x)=[x,t]$, $t\geq0$.
     We have 
     \[
     X=\bigcup _{k\in \mathbb{N}}f_k(U),\;\; \text{and}\;\; f_k(U)\subset f_{k+1}(U)\;\forall k\in \mathbb{N}.
     \]
      Therefore, $X$ has a complex manifold structure and each $f_t:U\rightarrow X$ is holomorphic one-one  map. This can be shown in exactly the same way as in our proof of Theorem \ref{Embedding of Steins having R_+ action into Steins having R actions}.
     We want to embed $X$ as a Stein Runge domain in $\Cpn$.
     Define an $\mathbb{R}$-action $\tilde{\theta}:\mathbb{R}\times X\rightarrow X$ by \[\tilde{\theta}(t,[x,s])=\begin{cases}
         [\theta_t(x),s]  \ \ \text{when} \ t\geq 0,\\
         [x,s-t]    \ \ \text{when} \ t\leq 0.
     \end{cases}\]\\
    This is well-defined, continuous and each $\tilde{\theta}_t:X\rightarrow X (t\in \mathbb{R})$ is holomorphic (see Theorem \ref{Embedding of Steins having R_+ action into Steins having R actions}).
     Choose $\Omega=\{[x,0]:x\in U\}\subset X$. For any $t\geq 0, \ [x,0]\in \Omega$, we have $\tilde{\theta}_t([x,0])=[\theta_t(x),0]\in \Omega$. Therefore  $\Omega$ is invariant under $\tilde{\theta}_t \ (\forall t\geq 0)$. $(\tilde{\theta}_t(\Omega),\Omega)$ is a Runge pair for each $t\geq 0$ (by Theorem \ref{Checking Runge pairs for R-actions}). $X$ is also a Stein manifold. $f_0^{-1}:\Omega\rightarrow U$ is a biholomorphic map onto $U$. Since $U$ is Runge in $\Cpn$ it follows from Theorem \ref{Embedding of Stein manifold using R-actions} that $f_0^{-1}$ can be approximated by elements of $ \mathcal{E}_{\mathcal{R}}(X,\Cpn)$ locally uniformly on $\Omega$. Therefore $ \mathcal{E}_{\mathcal{R}}(X,\Cpn)\neq \emptyset$ and $X$ has a holomorphic Runge embedding in $\Cpn$.  
\end{example}

\section{$\cplx^n$-valued solution of the Loewner PDE}
We now prove our main theorem on the Loewner PDE. We see the Loewner range as an increasing union of Stein domains, and in each step we construct a continuous isotopy to use Theorem \ref{Main theorem for embedding Stein manifolds}. We choose the Stein manifold $W$ as a domain with the density property in $\Cpn$. It allows us to get an embedding into $\Cpn$. 
\begin{proof}[Proof of Theorem \ref{Solution of Loewner on holomorphically convex domains in density property}:]
Using Result \ref{Starlike-existence of phi},
    on a complete hyperbolic domain $D$ in $\Cpn$, any Herglotz vector field $G(z,t)$ of order $d\in[1,+\infty]$ admits a unique evolution family $(\phi_{s,t}$) of order $d$ over $D$ such that, for all $z\in D$,
    \begin{equation*}\label{Loewner ODE1}
        \frac{\partial\phi_{s,t}}{\partial t}(z)=G(\phi_{s,t}(z),t) \ \ a.e. \ t\in [s,+\infty).
    \end{equation*}
    Again, by the Result \ref{Starlike-existence of univalent family for Loewner PDE}, there always exists a family of univalent mappings $(g_t:D\rightarrow N)$ satisfying \begin{equation}\label{evolution family associated to Loewner chain}
         g_s=g_t\circ\phi_{s,t}, 0\leq s\leq t \ \ \text{on} \ D,
    \end{equation}
    with values in a $n$-dimensional complex manifold $N$, which depends on $G(z,t)$.
    From equation (\ref{evolution family associated to Loewner chain}), it follows that $g_s(D)\subset g_t(D) (0\leq s\leq t).$
    We can rewrite the equation (\ref{evolution family associated to Loewner chain}) as \begin{equation}\label{evolution family associated to Loewner chain modified}
        g^{-1}_t\circ g_s=\phi_{s,t} \ \text{on}\ D \ (0\leq s\leq t).
    \end{equation} Also, using Result \ref{Solution of Loewner PDE with values in a manifold}, we know that any other solution to (\ref{Loewner PDE1}) with values in a complex manifold $Q$ is of the form $(\Lambda\circ g_t)$, where $\Lambda:N \rightarrow Q$ is holomorphic.
    Our aim is to embed the Loewner range $N$ into $\Cpn$.
    Choose any $k,m\in \mathbb{N}\cup\{0\}=\mathbb{N}_0$ such that $k<m$, and let us construct an isotopy, $H:[0,1]\times g_k(D)\rightarrow g_k(D)$ by 
    $$H(t,w)=g_k\circ g_{(1-t)k+tm}^{-1}(w), \ \forall w\in g_k(D).$$
    For  $w\in g_k(D)$, we obtain that
    \begin{align*}
        H(t,w)=&g_k\circ g_{(1-t)k+tm}^{-1}\circ g_k\circ g_k^{-1}(w)\\
               =& g_k\circ (g_{(1-t)k+tm}^{-1}\circ g_k)\circ g_k^{-1}(w)\\
               =& g_k\circ \phi_{k,(1-t)k+tm}\circ g_k^{-1}(w) \ (\text{by equation (\ref{evolution family associated to Loewner chain modified})})
    \end{align*}
    Using Result \ref{continuity of phis,t}, we get that $\phi_{k,(1-t)k+tm}$ is a continuous family in $t\in [0,1]$. 
    Therefore, $H$ is continuous on $[0,1]\times g_k(D)$  and $H(0,w)=w$.
    We also obtain that $H(1,\cdot)=H_1: g_k(D)\rightarrow g_k(D)$ can be extended to an injective holomorphic map \[\tilde{H}_1:g_m(D)\rightarrow g_k(D), \text{ given by } \tilde{H}_1(w)=g_k\circ g_m^{-1}(w).\] It follows from the fact that $\tilde{H}_{1}=H_1$ on $g_k(D)$.
   Since $(g_k(D),g_k(D))$ is a Runge pair and $\tilde{H}_1(g_m(D))=g_k(D)$, it follows that $(\tilde{H}_{1}(g_m(D)),g_k(D))$ is a Runge pair.
 Next, we claim that $(H_{t}(g_k(D)),g_k(D))$ is a Runge pair for each $t\in [0,1]$ and $k\in \mathbb{N}_0$.
By Result \ref{Runge pairs in Loewner range Hamada}, we know that for $s\geq0$, and each $k\in \mathbb{N}_0$, $(g_k(D),g_{k+s}(D))$ is a Runge pair.
We now apply the injective holomorphic map $g_{k+s}^{-1}:g_{k+s}(D)\rightarrow D$ on the Runge pair $(g_k(D),g_{k+s}(D))$. Using Lemma \ref{Rungepair biholomorphic invariant}, we get that $(g_{k+s}^{-1}\circ g_k(D),D)$ is a Runge pair, which is equivalent to the fact that $(\phi_{k,k+s}(D),D)$ is a Runge pair. Again we apply the injective holomorphic map  $g_k:D\rightarrow g_k(D)$,
and we get that $(g_k\circ \phi_{k,k+s}(D),g_k(D))$ is a Runge pair in $N$. Now \[H_t=g_k\circ \phi_{k,(1-t)k+tm}\circ g_k^{-1} \ \text{on}\  g_k(D)\Rightarrow H_t\circ g_k=g_k\circ \phi_{k,(1-t)k+tm} \ \text{on} \ D.\]
Hence $(H_t\circ g_k(D),g_k(D))$ is a Runge pair for each $t\in [0,1]$. Therefore we conclude that for each pair $(g_k(D),g_{k+1}(D))\ (k\in \mathbb{N}_0)$ we can find an isotopy of injective holomorphic maps that satisfies assumptions $(i),(ii),(iii)$ of Theorem \ref{Main theorem for embedding Stein manifolds}.  \\
Now, for each $j\in \mathbb{N}$, we write $\Omega_j=g_{j-1}(D)$. Since $D$ is complete hyperbolic domain, we get that $D$ is Stein (see Theorem $3.4$, Chapter V in \cite{Kobayashi}). Again $g_{j-1}:D\rightarrow \Omega_j$ is an one-one, onto holomorphic map. Hence $\Omega_j$ is also a Stein domain in $N$. Now we know that $\Omega_j\subset\Omega_{j+1}$, $\cup_{j\in \mathbb{N}}\Omega_j=N$, and $(\Omega_j,\Omega_{j+1})$ is Runge pair for each $j\in \mathbb{N}$. By Result \ref{Stein property of increasing union}, $N$ is a Stein manifold. 
From the assumption, there exists an injective holomorphic map $\Psi:D\rightarrow W$ such that $\Psi(D)$ is Runge in $W$. Consider the holomorphic embedding $\Psi_0=\Psi\circ g_0^{-1}:g_0(D)\rightarrow W$. Then $\Psi_0(g_0(D))$ is Runge in $W$. This implies that ${\mathcal{E}}_{\mathcal{R}}(g_{0}(D),W)\neq \emptyset$. Therefore $\{\Omega_j\}_{j\in \mathbb{N}}, N$ satisfies all the assumption to apply Theorem \ref{Main theorem for embedding Stein manifolds}. 
Now, using Theorem \ref{Main theorem for embedding Stein manifolds}, we conclude that the Loewner range $N=\cup_{k\in \mathbb{N}_0} g_k(D)$ can be embedded as a Stein and Runge domain in $W$. Assume that $F:N\rightarrow\Cpn$ is the corresponding holomorphic embedding map into $\Cpn$. Then $f_t=F\circ g_t: D\rightarrow\Cpn$ is the solution of the Loewner PDE 
\begin{equation*}
     \frac{\partial f_{t}}{\partial t}(z)=-df_t(z)G(z,t),   \ a.e. \  t\geq 0, \ \forall z\in D,
\end{equation*}
and $R=F(N)$ is the Loewner range in $\Cpn$. Any other solution with values in $\Cpn$ is of the form $(\Lambda\circ F^{-1})\circ f_t$ where  $\Lambda\circ F^{-1}:F(N)\rightarrow Q=\Cpn$.\\
\end{proof}

\begin{proof}[Proof of Corollary \ref{cor-Solution of Loewner on complete hyperbolic domains that admits special vector fields}] 
    From the given $G_{t_0}$ and Proposition \ref{Runge embedding of phi-like domains}, it follows that there exists a Runge embedding of $D$, say $\Phi:D\rightarrow \Cpn$. Now we use Theorem \ref{Solution of Loewner on holomorphically convex domains in density property} and the result follows.  
\end{proof}
  Next, we provide the first example of a complete hyperbolic domain which is not Runge in $\Cpn$ and non-contractible, but the Loewner PDE admits $\Cpn$-valued solution.
 \begin{example}\label{Example-Solution of Loewner on non-Runge domains}
     Consider the analytic polyhedron $D=\{(z_1,z_2,z_3)\in \mathbb{C}^*\times \mathbb{C}^2: |z_1|<5,|z_2|<5,|z_3|<5,|\frac{1}{z_1}|<1,|p(z_2,z_3)|<1\}$, where $p$ is any non-constant homogeneous polynomial in $z_2,z_3$ such that $D\neq \emptyset$.
     Then, $D$ is complete hyperbolic domain as it is an analytic polyhedra (see Chapter IV, Corollary 4.5 in \cite{Kobayashi}).  Clearly, $D$ is Runge in $\mathbb{C}^*\times \mathbb{C}^2$. The holomorphic function $h:D\rightarrow\mathbb{C}, \ h(z_1,z_2,z_3)=\frac{1}{z_1}$ is not a uniform limit of polynomials on $\{(2e^{i\theta},0,0):\theta\in\mathbb{R}\}\subset D$ . Therefore, $D$ is not Runge in $\mathbb{C}^3$. Consider the vector field $G:D\times \mathbb{R}^+\rightarrow \mathbb{C}^3$  \[
     G((z_1,z_2,z_3),t)=(iz_1,-z_2,-z_3).
     \]
     For each $t\in [0,\infty)$, 
     \[
     x(s)=(e^{is}a_1,e^{-s}a_2,e^{-s}a_3), \;\;s\in[ 0,\infty),
     \]
     is the solution of the Cauchy problem  \[\frac{d}{ds}x(s)=G(x(s),t), \ x(0)=(a_1,a_2,a_3).\] It is easy to check that $G$ is a Herglotz vector field on $D$. Since $\mathbb{C}^*\times \mathbb{C}^2$ has the density property  (see \cite{Varolin}), hence, by Theorem \ref{Solution of Loewner on holomorphically convex domains in density property}, we conclude that the following Loewner PDE
    \begin{equation*}
      \frac{\partial f_{t}}{\partial t}(z_1,z_2,z_3)=-df_t(z_1,z_2,z_3)(iz_1,-z_2,-z_3),   \ a.e. \  t\geq 0, \ \forall z\in D, 
     \end{equation*}
   admits a $\mathbb{C}^*\times \mathbb{C}^2$-valued solution
   $f_t:D\rightarrow \mathbb{C}^*\times \mathbb{C}^2\ (t\geq0)$. 
    
 \end{example}

 \noindent {\bf Acknowledgements.} Sushil Gorai is partially supported by an ARG MATRICS grant (ANRF/ARGM/2025/002958/MTR).
 Gourab Paul is supported by CSIR Senior Research Fellowship  (09/0921(17090)/2023-EMR-I).

\end{document}